\documentclass[12pt,a4paper,reqno]{amsart}

\usepackage{amsmath,amsfonts,amsthm,amsopn,amssymb,cite,mathrsfs}
\usepackage{cancel,marginnote}

\usepackage{color,enumitem,graphicx}
\usepackage[colorlinks=true,urlcolor=blue,
citecolor=red,linkcolor=blue,linktocpage,pdfpagelabels]{hyperref}
\usepackage[english]{babel}

\usepackage[left=2.8cm,right=2.8cm,top=2.8cm,bottom=2.8cm]{geometry}

\makeindex

\numberwithin{equation}{section}
\newtheorem{theorem}{Theorem}[section]
\newtheorem{lemma}{Lemma}[section]

\newtheorem{remark}{Remark}[section]

\usepackage[hyperpageref]{backref}

\title[Schr\"{o}dinger equation for the in-between critical exponents case]
{Explicit formula of the critical mass and the energy ground state solution for the mixed Local-nonlocal Schr\"{o}dinger equation for the in-between critical exponents case}

\thanks{This work is supported by the Natural Science Research Project of Anhui Educational Committee (Grant No. 2023AH040155).}


\begin{document}
\maketitle

\centerline{\scshape Yu Su}
\medskip
{\footnotesize
\centerline{School of Mathematics and Big Data, Anhui University of Science and Technology}
\centerline{Huainan, Anhui 232001, China}
\centerline{yusumath@aust.edu.cn}
}

\centerline{\scshape Hichem Hajaiej}
\medskip
{\footnotesize
\centerline{Department of Mathematics, California State University at Los Angeles}
\centerline{Los Angeles, California 90032, USA}
\centerline{hhajaie@calstatela.edu}
}

\begin{abstract}
In this paper,
we study the Schr\"{o}dinger equation with the mixed local-nonlocal operator in the presence of a power nonlinearity $u|u|^{p-2}$, when the exponent $p$ is between the two critical exponents $2+\frac{4s}{N}$(of the fractional Laplacian), and $2+\frac{4}{N}$( of the Laplacian). The equation presents a new phenomena that we will call here ``the in-between critical exponents" $p\in(2+\frac{4s}{N},2+\frac{4}{N}).$
In this case, there exists a critical mass $c_{0}$, such that for any mass $c<c_{0}$ there is no energy ground state solution, and for any mass $c\geqslant c_{0}$ there exists an energy ground state solution.

Our first main contribution consists in establishing an explicit formula of the critical mass $c_{0}$ via the best constant of the Gagliardo-Nirenberg inequality for the mixed local-nonlocal Laplacian.

We also prove the existence of an optimizer of the Gagliardo-Nirenberg inequality for the mixed local-nonlocal operator.
We then show that the optimizer (after some suitable scaling) is an energy ground state solution $u_{c_{0}}$ (with critical mass $c_{0}$).

This is a key ingredient to determine sufficient and necessary conditions of existence and non-existence of energy ground state solutions in the in-between critical exponents case.
Finally,
we show that the energy ground state solution $u_{c_{0}}$ is an optimizer of the Gagliardo-Nirenberg inequality for the mixed local-nonlocal operator.
\end{abstract}

\section{Introduction}
We study the following  Schr\"{o}dinger equation with a mixed local-nonlocal operator
\begin{equation}\label{1.1}
\begin{aligned}
\begin{cases}
i\frac{\partial \psi}{\partial t}
+\Delta \psi
-(-\Delta)^{s}\psi=|\psi|^{p-2}\psi, \ \ (t,x)\in \mathbb{R}^{+}\times\mathbb{R}^{N},\\
\psi(0,x)=\psi_{0}(x),
\end{cases}
\end{aligned}
\end{equation}
where $N\geqslant3$, $s\in(0,1)$, $2+\frac{4s}{N}<p<2+\frac{4}{N}$,
and $(-\Delta)^{s}$ is the so-called fractional
Laplacian, which can be defined, for any $u:\mathbb{R}^{N}\to\mathbb{R}$ smooth enough, by setting
$
\mathcal{F}((-\Delta)^{s}u)(\xi)
=
|\xi|^{2s}
\mathcal{F}(u)(\xi)
$,
$\xi\in\mathbb{R}^{N}$,
where $\mathcal{F}$ represents the Fourier transform, see \cite{Nezza-Palatucci-Valdinoci2012BSM}.

Equation \eqref{1.1} arises in the population dynamics model with both classical and nonlocal diffusion see Dipierro-Lippi-Valdinocci \cite{Dipierro-Lippi-Valdinocci2022AIHP}.
Moreover, Biagi-Dipierro-Valdinoci-Vecchi \cite{Biagi-Dipierro-Valdinoci-Vecchi2021CPDE} pointed out that equation \eqref{1.1} can apply to study the different types of ``regional" or ``global" restrictions that may reduce the spreading of a pandemic disease.
As explained in \cite{Dipierro-Valdinocci2021PA} by Dipierro-Valdinocci,
equation \eqref{1.1} also describes an ecological niche for a mixed local and nonlocal dispersal.
For a more detailed account, we refer the reader to Bernstein-type regularity results
\cite{Cabre-Dipierro-Valdinoci2022ARMA},
the Aubry-Mather theory for sums of different fractional Laplacians
\cite{Llave-Valdinoci2009Poincare}, numerics\cite{Biswas-Jakobsen-Karlsen2010SIAM-J-N-A},
probability and stochastics
\cite{Mimica2016PLMS,ChenZQ-Kim-Song-Vondracek2012TAMS}.
The sum and the difference of two or more fractional Laplacians appear
in many other fields, we refer the reader to page 2 of Chen-Bhakta-Hajaiej \cite{ChenHY-Bhakta-Hajaiej2022JDE},
and also refer to \cite{Hajaiej-Perera2022DIE,Maione-Mugnai-Vecchi2023FFAC,Cangiotti-Caponi-Maione-Vitillaro2023Milan,Cangiotti-Caponi-Maione-Vitillaro2024FFAC,Giovannardi-Mugnai-Vecchi2023JMAA}.

The standing waves to equation \eqref{1.1} are
\begin{equation*}
\begin{aligned}
\psi(t,x)=e^{i\omega t}u(x),
\end{aligned}
\end{equation*}
where $\omega \in \mathbb{R}$. It is straightforward to see that a standing wave $\psi$ is a solution to equation \eqref{1.1} if and only if $u$ is a solution to the following  Schr\"{o}dinger equation 
\begin{equation}\label{MLN}
\begin{aligned}
-\Delta u
+(-\Delta)^{s}u+\omega u=|u|^{p-2}u, \ \ x\in \mathbb{R}^{N}.
\end{aligned}
\tag{MLN}
\end{equation}
If $u$ is a weak solution of equation \eqref{MLN},
then for any $\varphi\in H^{1}(\mathbb{R}^{N})$ it satisfies
\begin{equation*}
\begin{aligned}
0=&
\int_{\mathbb{R}^{N}}
\nabla u
\nabla \varphi
\mathrm{d}x
+
\int_{\mathbb{R}^{N}}
\int_{\mathbb{R}^{N}}
\frac{(u(x)-u(y))(\varphi(x)-\varphi(y))}{|x-y|^{N+2s}}
\mathrm{d}x
\mathrm{d}y
+\omega
\int_{\mathbb{R}^{N}}
u\varphi
\mathrm{d}x\\
&-
\int_{\mathbb{R}^{N}}
|u|^{p-2}
u\varphi
\mathrm{d}x.
\end{aligned}
\end{equation*}
Su-Valdinoci-Wei-Zhang
\cite{Su-Valdinoci-Wei-Zhang2022MZ,Su-Valdinoci-Wei-Zhang2025JDE}
and Dipierro-Su-Valdinoci-Zhang
\cite{Dipierro-Su-Valdinoci-Zhang2025DCDS} studied the regularity of solutions to equation \eqref{MLN}.
In this current paper,
we address another highly important aspect.
We are looking for standing waves with prescribed mass for equation \eqref{MLN}.
From a physical point of view,
the most interesting solutions,
the so-called {\it energy ground state solution},
are the minimizers of the problem
\begin{equation*}
\begin{aligned}
m_{c}:=\inf_{u\in S_{c}}J(u),
\end{aligned}
\end{equation*}
where
\begin{equation*}
\begin{aligned}
S_{c}:=\left\{u\in H^{1}(\mathbb{R}^{N})\bigg| \int_{\mathbb{R}^{N}}
|u|^{2}
\mathrm{d}x
=c\right\}.
\end{aligned}
\end{equation*}
Here, the energy functional is
\begin{equation*}
\begin{aligned}
J(u):=\frac{1}{2}\int_{\mathbb{R}^{N}}
|\nabla u|^{2}
\mathrm{d}x
+\frac{1}{2}
\int_{\mathbb{R}^{N}}
\int_{\mathbb{R}^{N}}
\frac{|u(x)-u(y)|^{2}}{|x-y|^{N+2s}}
\mathrm{d}x
\mathrm{d}y
-
\frac{1}{p}
\int_{\mathbb{R}^{N}}
|u|^{p}
\mathrm{d}x.
\end{aligned}
\end{equation*}
We know that each minimizer $u\in S_{c}$ of $m_{c}$ can be associated to a Lagrange multiplier $\omega>0$ such that
$(u,\omega)$ weakly solves \eqref{MLN}.

There are two important exponents related to  \eqref{MLN}:
\begin{equation*}
\begin{aligned}
2+\frac{4s}{N},
\end{aligned}
\end{equation*}
and
\begin{equation*}
\begin{aligned}
2+\frac{4}{N},
\end{aligned}
\end{equation*}
where the exponent $2+\frac{4}{N}$ is the mass-critical exponent for the classical Laplacian problem,
and the exponent $2+\frac{4s}{N}$ is the mass-critical exponent for the fractional Laplacian problem.
In the mixed local-nonlcoal operator setting,
the exponent $2+\frac{4}{N}$ is the mass critical exponent.
Although the exponent $2+\frac{4s}{N}$ is not the mass-critical exponent in the mixed local-nonlcoal operator problem, it brings a new phenomenon:
The in-between critical exponents case
\begin{equation*}
\begin{aligned}
2+\frac{4s}{N}<p<2+\frac{4}{N}.
\end{aligned}
\end{equation*}
This new phenomenon arises from the ``conflict" between the local operator and nonlocal operator, see Theorem \ref{Theorem1.1}.

Luo-Hajaiej \cite{LuoTJ-Hajaiej2022ANS} considered the following nonlinear Schr\"odinger equation with the mixed fractional Laplacians ($0<s_{1}<s_{2}<1$)
\begin{equation}\label{MF}
\begin{aligned}
(-\Delta)^{s_{1}} u
+(-\Delta)^{s_{2}}u+\omega u=|u|^{p-2}u, \ \ x\in \mathbb{R}^{N}.
\end{aligned}
\tag{MF}
\end{equation}
When $s_{2}=1$,
 equation \eqref{MF} reduces to equation \eqref{MLN}.
Let us recall the main results  of \cite{LuoTJ-Hajaiej2022ANS}.

\noindent
{\bf Theorem A} \cite[Theorem  1.2]{LuoTJ-Hajaiej2022ANS}
Let $N\geqslant1$ and $0<s_{1}<s_{2}<1$.
\begin{enumerate}
\item
if $p\in(2,2+\frac{4s_{1}}{N})$,
then equation \eqref{MF} has an energy ground state solution for any $c>0$;

\item
if $p=2+\frac{4s_{1}}{N}$,
then there exists $\bar{c}_{1}>0$ such that equation \eqref{MF} has an energy ground state solution if  $c>\bar{c}_{1}$,
and  equation \eqref{MF} has no energy ground state solution if  $c<\bar{c}_{1}$;

\item
$p\in(2+\frac{4s_{1}}{N},2+\frac{4s_{2}}{N})$,
then there exists $\bar{c}_{0}>0$ such that equation \eqref{MF} has an energy ground state solution if and only if $c\geqslant \bar{c}_{0}$;

\item
$p=2+\frac{4s_{2}}{N}$,
then equation \eqref{MF} has no energy ground state solution for any $c>0$.
\end{enumerate}
For $p\in(2+\frac{s_{2}}{N},\frac{2N}{N-2s_{2}})$,
Chergui-Gou-Hajaiej \cite{Chergui-Gou-Hajaiej2023CVPDE} investigated the existence, multiplicity and orbital instability results for equation \eqref{MF}.

To the best of our knowledge,
there are no results addressing the energy ground state solution for \eqref{MLN}.
With minor changes to \cite[Theorems 1.1 and 1.2]{LuoTJ-Hajaiej2022ANS},
one has the following result about the in-between critical exponents case.

\begin{theorem}\label{Theorem1.1}
Let $N\geqslant3$, $s\in(0,1)$ and  $p\in(2+\frac{4s}{N},2+\frac{4}{N})$.
Then there exists $c_{0}>0$ such that equation \eqref{MLN} has an energy ground state solution if and only if $c\geqslant c_{0}$.
\end{theorem}

Theorem \ref{Theorem1.1} only states the existence of the critical mass $c_{0}$, and its corresponding energy ground state solution $u_{c_{0}}$.
The main goal of this paper is to provide an explicit formula of $c_{0}$, and to establish some qualitative and quantitative properties of $u_{c_{0}}$.

\subsection{Main Results}
In order to characterize the critical mass $c_{0}$ and the energy ground state solution $u_{c_{0}}$,
we start with the Gagliardo-Nirenberg inequality for the mixed local-nonlocal operator:
\begin{theorem}\label{Theorem1.2}
Let $N\geqslant3$,
$s\in(0,1)$
and $p\in(2+\frac{4s}{N},2+\frac{4}{N})$.
For $u\in H^{1}(\mathbb{R}^{N})$,
there exists a constant $C_{1,s,p}>0$ such that
\begin{equation}\label{1.2}
\begin{aligned}
\|u\|_{L^{p}(\mathbb{R}^{N})}^{p}
\leqslant
C_{1,s,p}
\|u\|_{D^{s,2}(\mathbb{R}^{N})}^{\frac{4-N(p-2)}{2(1-s)}}
\|u\|_{D^{1,2}(\mathbb{R}^{N})}^{\frac{N(p-2)-4s}{2(1-s)}}
\|u\|_{L^{2}(\mathbb{R}^{N})}^{p-2}.
\end{aligned}
\end{equation}
Also,
there exists a non-negative radial optimizer $Q\in H^{1}(\mathbb{R}^{N})$ for inequality \eqref{1.2},
and it is a weak solution of
\begin{equation*}
\begin{aligned}
-\Delta Q
+(-\Delta)^{s}Q+\omega Q=|Q|^{p-2}Q, \ \ x\in \mathbb{R}^{N},
\end{aligned}
\end{equation*}
with $\omega>0$.
\end{theorem}

By virtue of Theorem \ref{Theorem1.2},
we prove the explicit formula of the best constant $C_{1,s,p}$.
Then we show the explicit formula of the critical mass $c_{0}$ via the best constant $C_{1,s,p}$.
Furthermore,
we prove that $Q$ is an energy ground state solution with critical mass $c_{0}$.
\begin{theorem}\label{Theorem1.3}
Let $N\geqslant3$,
$s\in(0,1)$
and $p\in(2+\frac{4s}{N},2+\frac{4}{N})$.
Then we have the following results.

\begin{enumerate}
\item
Let $c_{0}$ be the critical mass.
Then
\begin{equation*}
\begin{aligned}
C_{1,s,p}
=\left[
\frac{2N+4-pN}{2p(1-s)}
\right]^{-\frac{2N+4-pN}{4(1-s)}}
\left[
\frac{pN-2N-4s}{2p(1-s)}
\right]^{-\frac{pN-2N-4s}{4(1-s)}}
\frac{1}{c_{0}^{\frac{p-2}{2}}},
\end{aligned}
\end{equation*}
meaning that:
\begin{equation*}
\begin{aligned}
c_{0}:=C_{1,s,p}^{-\frac{2}{p-2}}
\left[
\frac{2N+4-pN}{2p(1-s)}
\right]^{-\frac{2N+4-pN}{2(1-s)(p-2)}}
\left[
\frac{pN-2N-4s}{2p(1-s)}
\right]^{-\frac{pN-2N-4s}{2(1-s)(p-2)}}.
\end{aligned}
\end{equation*}

\item
We have
\begin{equation*}
\begin{aligned}
J(Q)=0,
\end{aligned}
\end{equation*}
and
\begin{equation*}
\begin{aligned}
\int_{\mathbb{R}^{N}}
|Q|^{2}
\mathrm{d}x=c_{0}.
\end{aligned}
\end{equation*}
Moreover,
$Q$ is an energy ground state solution with critical mass $c_{0}$.
\end{enumerate}

\end{theorem}

\begin{remark}
Combining Theorems \ref{Theorem1.2} and \ref{Theorem1.3},
we know that if $Q$ is an optimizer of inequality \eqref{1.2},
and also a weak solution of equation \eqref{MLN},
then $Q$ is an energy ground state solution with critical mass $c_{0}$.

We will first show that $u_{c_{0}}$ is an optimizer of inequality \eqref{1.2}.
\end{remark}

A preparatory step is to investigate the sufficient and necessary conditions for $m_{c}<0$ and $m_{c}=0$.
\begin{theorem}\label{Theorem1.4}
Let $N\geqslant3$, $s\in(0,1)$ and  $p\in(2+\frac{4s}{N},2+\frac{4}{N})$. Set
\begin{equation*}
\begin{aligned}
m_{c}:=\inf_{u\in S_{c}} J(u).
\end{aligned}
\end{equation*}
Then the following three conditions are equivalent
\begin{enumerate}
\item [$(a)$] $m_{c}<0$;

\item [$(b)$] $c>c_{0}$;

\item [$(c)$] there exists $u\in S_{c}$ such that
\begin{equation*}
\begin{aligned}
\|u\|_{D^{1,2}(\mathbb{R}^{N})}^{\frac{N(p-2)-4s}{2-2s}}
\|u\|_{D^{s,2}(\mathbb{R}^{N})}^{\frac{4-N(p-2)}{2-2s}}
<
\frac{N(p-2)-4s}{2p(1-s)}
\left(
\frac{4-N(p-2)}{N(p-2)-4s}
\right)^{\frac{4-N(p-2)}{4-4s}}
\int_{\mathbb{R}^{N}}
|u|^{p}
\mathrm{d}x.
\end{aligned}
\end{equation*}
\end{enumerate}

\noindent The following three conditions are equivalent
\begin{enumerate}
\item [$(a')$]
$m_{c}=0$;

\item [$(b')$]
$c\leqslant c_{0}$;

\item [$(c')$]
for all $u\in S_{c}$,
\begin{equation*}
\begin{aligned}
\|u\|_{D^{1,2}(\mathbb{R}^{N})}^{\frac{N(p-2)-4s}{2-2s}}
\|u\|_{D^{s,2}(\mathbb{R}^{N})}^{\frac{4-N(p-2)}{2-2s}}
\geqslant
\frac{N(p-2)-4s}{2p(1-s)}
\left(
\frac{4-N(p-2)}{N(p-2)-4s}
\right)^{\frac{4-N(p-2)}{4-4s}}
\int_{\mathbb{R}^{N}}
|u|^{p}
\mathrm{d}x.
\end{aligned}
\end{equation*}
\end{enumerate}
\end{theorem}
Using Theorem \ref{Theorem1.4} we can state that $u_{c_{0}}$ is an optimizer of inequality \eqref{1.2}.
\begin{theorem}\label{Theorem1.5}
Let $N\geqslant3$,
$s\in(0,1)$
and  $p\in(2+\frac{4s}{N},2+\frac{4}{N})$.
Let $u_{c_{0}}\in H^{1}(\mathbb{R}^{N})$ be an energy ground state solution for equation \eqref{MLN} with $c=c_{0}$.
Then
\begin{equation*}
\begin{aligned}
\|u_{c_{0}}\|_{D^{1,2}(\mathbb{R}^{N})}^{\frac{N(p-2)-4s}{2-2s}}
\|u_{c_{0}}\|_{D^{s,2}(\mathbb{R}^{N})}^{\frac{4-N(p-2)}{2-2s}}
=C_{1,s,p}^{-1}\|u_{c_{0}}\|_{L^{2}(\mathbb{R}^{N})}^{-(p-2)}
\int_{\mathbb{R}^{N}}
|u_{c_{0}}|^{p}
\mathrm{d}x.
\end{aligned}
\end{equation*}
Moreover,
we know that $u_{c_{0}}$ is an optimizer of inequality \eqref{1.2}.
\end{theorem}

We point out that our method applies to equation \eqref{MF}.
We obtain the explicit formula of the critical mass $\bar{c}_{0}$,
and the relation between the energy ground state solution and the optimizer of the Gagliardo-Nirenberg inequality with mixed fractional operators as follows
\begin{equation}\label{1.3}
\begin{aligned}
\|u\|_{L^{p}(\mathbb{R}^{N})}^{p}
\leqslant C_{s_{2},s_{1},p}
\|u\|_{D^{s_{2},2}(\mathbb{R}^{N})}^{\frac{4-N(p-2)}{2(1-s)}}
\|u\|_{D^{s_{1},2}(\mathbb{R}^{N})}^{\frac{N(p-2)-4s}{2(1-s)}}
\|u\|_{L^{2}(\mathbb{R}^{N})}^{p-2},
\end{aligned}
\end{equation}
where $u\in H^{s_{2}}(\mathbb{R}^{N})$
and $C_{s_{2},s_{1},p}>0$ is the best constant of inequality \eqref{1.3}.

For the convenience of the readers,
we state the result as follows.
\begin{theorem}\label{Theorem1.6}
Let $N\geqslant1$,
$0<s_{1}<s_{2}<1$
and $p\in(2+\frac{4s_{1}}{N},2+\frac{4s_{2}}{N})$.
Then the following results hold.
\begin{enumerate}
\item
There exists a non-negative radial optimizer $\bar{Q}\in H^{s_{2}}(\mathbb{R}^{N})$ for inequality \eqref{1.3}, and it is the weak solution of
\begin{equation*}
\begin{aligned}
(-\Delta)^{s_{2}} \bar{Q}
+(-\Delta)^{s_{1}}\bar{Q}+\omega \bar{Q}=|\bar{Q}|^{p-2}\bar{Q}, \ \ x\in \mathbb{R}^{N},
\end{aligned}
\end{equation*}
with $\omega>0$.
\item
Let $\bar{c}_{0}$ be the critical mass in Theorem A.
Then
\begin{equation*}
\begin{aligned}
C_{s_{2},s_{1},p}
=
\left[
\frac{2N+4s_{2}-pN}{2p(s_{2}-s_{1})}
\right]^{-\frac{2N+4s_{2}-pN}{4(s_{2}-s_{1})}}
\left[
\frac{pN-2N-4s_{1}}{2p(s_{2}-s_{1})}
\right]^{-\frac{pN-2N+4s_{1}}{4(s_{2}-s_{1})}}
\frac{1}{\bar{c}_{0}^{\frac{p-2}{2}}}
\end{aligned}
\end{equation*}
and
\begin{equation*}
\begin{aligned}
\bar{c}_{0}
=
\left[
\frac{2N+4s_{2}-pN}{2p(s_{2}-s_{1})}
\right]^{\frac{2N+4s_{2}-pN}{2(s_{2}-s_{1})(p-2)}}
\left[
\frac{pN-2N-4s_{1}}{2p(s_{2}-s_{1})}
\right]^{\frac{pN-2N+4s_{1}}{2(s_{2}-s_{1})(p-2)}}
C_{s_{2},s_{1},p}^{-\frac{2}{p-2}}
\end{aligned}
\end{equation*}
\item
Let $u_{\bar{c}_{0}}\in H^{s_{2}}(\mathbb{R}^{N})$ be an energy ground state solution for equation \eqref{MF} with $c=\bar{c}_{0}$.
Then
\begin{equation*}
\begin{aligned}
\|u_{\bar{c}_{0}}\|_{D^{s_{2},2}(\mathbb{R}^{N})}^{\frac{4-N(p-2)}{2(1-s)}}
\|u_{\bar{c}_{0}}\|_{D^{s_{1},2}(\mathbb{R}^{N})}^{\frac{N(p-2)-4s}{2(1-s)}}
=C_{s_{2},s_{1},p}^{-1}\|u_{\bar{c}_{0}}\|_{L^{2}(\mathbb{R}^{N})}^{-(p-2)}
\int_{\mathbb{R}^{N}}
|u_{\bar{c}_{0}}|^{p}
\mathrm{d}x.
\end{aligned}
\end{equation*}
Moreover,
we know that $u_{\bar{c}_{0}}$ is an optimizer of inequality \eqref{1.3}.
\end{enumerate}
\end{theorem}

\noindent
{\bf Structure of this paper}: In Section 2, we present some  preliminary results about Sobolev spaces.
In Section 3, we prove Theorem \ref{Theorem1.1}.
In Section 4,
we show Theorem \ref{Theorem1.2}.
In Section 5,
we prove Theorem \ref{Theorem1.3}.
In Section 6, we show Theorem \ref{Theorem1.4}.
In Section 7, we prove Theorem \ref{Theorem1.5}.
In Section 8, we give the proof of Theorem  \ref{Theorem1.6}.

\section{Sobolev Spaces}
Define the homogeneous Sobolev space
\begin{equation*}
\begin{aligned}
D^{1,2}(\mathbb{R}^{N})
=
\{u\in L^{\frac{2N}{N-2}}(\mathbb{R}^{N})|
|\nabla u|\in L^{2}(\mathbb{R}^{N})\},
\end{aligned}
\end{equation*}
its semi-norm is defined as
\begin{equation*}
\begin{aligned}
\|u\|_{D^{1,2}
(\mathbb{R}^{N})}^{2}
=\int_{\mathbb{R}^{N}}
|\nabla u|^{2}
\mathrm{d}x.
\end{aligned}
\end{equation*}
For $N\geqslant 3$ and $s\in(0,1)$, let $D^{s,2}(\mathbb{R}^{N})$ be the homogeneous fractional Sobolev space, which is the completion of $C_{0}^{\infty}(\mathbb{R}^{N})$ with the semi-norm
\begin{equation*}
\begin{aligned}
\|u\|_{D^{s,2}(\mathbb{R}^{N})}^{2}:=
\int_{\mathbb{R}^{N}}\int_{\mathbb{R}^{N}}\frac{|u(x)-u(y)|^{2}}{|x-y|^{N+2s}}\mathrm{d}x\mathrm{d}y.
\end{aligned}
\end{equation*}
Define the inhomogeneous Sobolev space by:
\begin{equation*}
\begin{aligned}
H^{1}(\mathbb{R}^{N})
=
\{u\in L^{2}(\mathbb{R}^{N})|
|\nabla u|\in L^{2}(\mathbb{R}^{N})\},
\end{aligned}
\end{equation*}
its norm is defined as
\begin{equation*}
\begin{aligned}
\|u\|_{H^{1}(\mathbb{R}^{N})}^{2}
=
\int_{\mathbb{R}^{N}}
|\nabla u|^{2}
\mathrm{d}x
+
\int_{\mathbb{R}^{N}}
|u|^{2}
\mathrm{d}x.
\end{aligned}
\end{equation*}
Define the inhomogeneous fractional Sobolev space by:
\begin{equation*}
\begin{aligned}
H^{s}(\mathbb{R}^{N})
=
\{u\in L^{2}(\mathbb{R}^{N})|
\|u\|_{D^{s,2}(\mathbb{R}^{N})}^{2}
<\infty\},
\end{aligned}
\end{equation*}
its norm is defined as
\begin{equation*}
\begin{aligned}
\|u\|_{H^{s}(\mathbb{R}^{N})}^{2}
=
\int_{\mathbb{R}^{N}}\int_{\mathbb{R}^{N}}\frac{|u(x)-u(y)|^{2}}{|x-y|^{N+2s}}\mathrm{d}x\mathrm{d}y
+
\int_{\mathbb{R}^{N}}
|u|^{2}
\mathrm{d}x.
\end{aligned}
\end{equation*}

\begin{lemma}[Continuous embedded]\label{Lemma2.1}
$H^{1}(\mathbb{R}^{N})
\hookrightarrow
D^{s,2}(\mathbb{R}^{N})$.
\end{lemma}
\begin{proof}
From \cite{Nezza-Palatucci-Valdinoci2012BSM},
H\"{o}lder's and Young's inequalities,
we have that
\begin{equation*}
\begin{aligned}
\|u\|_{D^{s,2}(\mathbb{R}^{N})}^{2}
=&
\int_{\mathbb{R}^{N}}
|\xi|^{2s}|\hat{u}(\xi)|^{2}
\mathrm{d}\xi\\
\leqslant&
\left(
\int_{\mathbb{R}^{N}}
|\xi|^{2}|\hat{u}(\xi)|^{2}
\mathrm{d}\xi
\right)^{\frac{2s}{2}}
\left(
\int_{\mathbb{R}^{N}}
|\hat{u}(\xi)|^{(2-2s)\frac{2}{2-2s}}
\mathrm{d}\xi
\right)^{\frac{2-2s}{2}}\\
=&
\left(
\int_{\mathbb{R}^{N}}
|\nabla u|^{2}
\mathrm{d}x
\right)^{\frac{2s}{2}}
\left(
\int_{\mathbb{R}^{N}}
|u|^{2}
\mathrm{d}x
\right)^{\frac{2-2s}{2}}\\
\leqslant&C\|u\|_{H^{1}(\mathbb{R}^{N})}^{2},
\end{aligned}
\end{equation*}
which proves the result.
\end{proof}

For $p\in(2,2^{*})$ and $u\in H^{1}(\mathbb{R}^{N})$, using the Gagliardo-Nirenberg inequality \cite{Weinstein1983CMP}, we know that there exists a constant $C>0$ such that
\begin{equation}\label{2.1}
\begin{aligned}
\|u\|_{L^{p}(\mathbb{R}^{N})}^{p}
\leqslant
C
\|u\|_{D^{1,2}(\mathbb{R}^{N})}^{\frac{N(p-2)}{2}}
\|u\|_{L^{2}(\mathbb{R}^{N})}^{p-\frac{N(p-2)}{2}}.
\end{aligned}
\end{equation}
For $q\in(2,\frac{2N}{N-2s})$ and $u\in H^{s}(\mathbb{R}^{N})$, we know that there exists a constant $C>0$ such that
\begin{equation}\label{2.2}
\begin{aligned}
\|u\|_{L^{q}(\mathbb{R}^{N})}^{q}
\leqslant
C
\|u\|_{D^{s,2}(\mathbb{R}^{N})}^{\frac{N(q-2)}{2s}}
\|u\|_{L^{2}(\mathbb{R}^{N})}^{q-\frac{N(q-2)}{2s}}.
\end{aligned}
\end{equation}

\section{Proof of Theorem \ref{Theorem1.1}}
In this section,
we prove Theorem \ref{Theorem1.1},
i.e.,
the existence and nonexistence of energy ground state solutions in the in-between critical exponents case.
\begin{lemma}\label{Lemma3.1}
Let $p\in(2+\frac{4s}{N},2+\frac{4}{N})$ and $c>c_{0}$.
Then we have, for all $c'>c$,
\begin{equation*}
\begin{aligned}
m_{c'}\leqslant \left[\frac{c'}{c}\right]^{\frac{2(p-2)}{4-N(p-2)}+1} m_{c}.
\end{aligned}
\end{equation*}
Furthermore,
\begin{equation*}
\begin{aligned}
m_{c}<m_{\bar{c}}+m_{c-\bar{c}},
\end{aligned}
\end{equation*}
for all $\bar{c}\in(0,c)$.
\end{lemma}

\begin{proof}
Let $u\in S_{c}$ and $\bar{u}(x):=(\frac{c'}{c})^{a}u((\frac{c'}{c})^{b}x)$,
where
\begin{equation*}
\begin{aligned}
a=\frac{2}{4-N(p-2)},~~b=\frac{p-2}{4-N(p-2)}.
\end{aligned}
\end{equation*}
Then we have that
\begin{equation*}
\begin{aligned}
\int_{\mathbb{R}^{N}}
|\bar{u}|^{2}
\mathrm{d}x
=&(c')^{2a-bN}c^{1-2a+bN}=c'.
\end{aligned}
\end{equation*}
A straigthfoward computation implies that:
\begin{equation*}
\begin{aligned}
\int_{\mathbb{R}^{N}}
|\nabla \bar{u}|^{2}
\mathrm{d}x
=&(\frac{c'}{c})^{2a+2b-bN}
\int_{\mathbb{R}^{N}}
|\nabla u(y)|^{2}
\mathrm{d}y,
\end{aligned}
\end{equation*}
\begin{equation*}
\begin{aligned}
\int_{\mathbb{R}^{N}}
\int_{\mathbb{R}^{N}}
\frac{|\bar{u}(x)-\bar{u}(y)|^{2}}{|x-y|^{N+2s}}
\mathrm{d}x
\mathrm{d}y
=&(\frac{c'}{c})^{2a+2bs-bN}
\int_{\mathbb{R}^{N}}
\int_{\mathbb{R}^{N}}
\frac{|u(x)-u(y)|^{2}}{|x-y|^{N+2s}}
\mathrm{d}x
\mathrm{d}y,
\end{aligned}
\end{equation*}
\begin{equation*}
\begin{aligned}
\int_{\mathbb{R}^{N}}
|\bar{u}|^{p}
\mathrm{d}x
=&(\frac{c'}{c})^{pa-bN}
\int_{\mathbb{R}^{N}}
|u|^{p}
\mathrm{d}x.
\end{aligned}
\end{equation*}
Then $\bar{u}(x)\in S_{c'}$, and since $2a+2b-bN=pa-bN$,
we get that
\begin{equation*}
\begin{aligned}
m_{c'}
\leqslant& J(\bar{u})\\
=&\frac{1}{2}
\int_{\mathbb{R}^{N}}
|\nabla \bar{u}|^{2}
\mathrm{d}x
+\frac{1}{2}
\int_{\mathbb{R}^{N}}
\int_{\mathbb{R}^{N}}
\frac{|\bar{u}(x)-\bar{u}(y)|^{2}}{|x-y|^{N+2s}}
\mathrm{d}x
\mathrm{d}y
-\frac{1}{p}
\int_{\mathbb{R}^{N}}
|\bar{u}|^{p}
\mathrm{d}x\\
=&(\frac{c'}{c})^{2a+2b-bN}
\frac{1}{2}
\int_{\mathbb{R}^{N}}
|\nabla u|^{2}
\mathrm{d}x
+(\frac{c'}{c})^{2a+2bs-bN}
\frac{1}{2}
\int_{\mathbb{R}^{N}}
\int_{\mathbb{R}^{N}}
\frac{|u(x)-u(y)|^{2}}{|x-y|^{N+2s}}
\mathrm{d}x
\mathrm{d}y\\
&-(\frac{c'}{c})^{pa-bN}\frac{1}{p}
\int_{\mathbb{R}^{N}}
|u|^{p}
\mathrm{d}x\\
=&(\frac{c'}{c})^{2a+2b-bN}
\left[
\frac{1}{2}
\int_{\mathbb{R}^{N}}
|\nabla u|^{2}
\mathrm{d}x
+
\frac{1}{2}
(\frac{c'}{c})^{2b(s-1)}
\int_{\mathbb{R}^{N}}
\int_{\mathbb{R}^{N}}
\frac{|u(x)-u(y)|^{2}}{|x-y|^{N+2s}}
\mathrm{d}x
\mathrm{d}y
-
\frac{1}{p}
\int_{\mathbb{R}^{N}}
|u|^{p}
\mathrm{d}x
\right]\\
\leqslant&(\frac{c'}{c})^{2a+2b-bN}J(u)
\end{aligned}
\end{equation*}
which gives, since $2a+2b-bN
=2b+1
=\frac{2(p-2)}{4-N(p-2)}+1$,
\begin{equation*}
\begin{aligned}
m_{c'}\leqslant \left[\frac{c'}{c}\right]^{\frac{2(p-2)}{4-N(p-2)}+1} m_{c}.
\end{aligned}
\end{equation*}
One has that
\begin{equation*}
\begin{aligned}
m_{c}
=&
\frac{c-\bar{c}}{c}m_{c}
+\frac{\bar{c}}{c}m_{c}\\
=&
\frac{c-\bar{c}}{c}m_{\frac{c}{c-\bar{c}}(c-\bar{c})}
+\frac{\bar{c}}{c}m_{\frac{c}{\bar{c}}\bar{c}}.
\end{aligned}
\end{equation*}
If $m_{\bar{c}}=0$ and $m_{c-\bar{c}}=0$, then we have $\bar{c}\leqslant c_{0}$ and $c-\bar{c}\leqslant c_{0}$. Therefore,
\begin{equation*}
\begin{aligned}
m_{c}<0=m_{\bar{c}}~~\mathrm{and}~~ m_{c}<0=m_{c-\bar{c}},
\end{aligned}
\end{equation*}
and then $m_{c}<m_{\bar{c}}+m_{c-\bar{c}}$.

If $m_{\bar{c}}<0$ and $m_{c-\bar{c}}=0$, then we have $\bar{c}>c_{0}$ and $c-\bar{c}\leqslant c_{0}$. And $m_{c}<0=m_{c-\bar{c}}$
and
\begin{equation*}
\begin{aligned}
m_{c}\leqslant (\frac{c}{\bar{c}})^{\frac{2(p-2)}{4-N(p-2)}+1}m_{\bar{c}}<m_{\bar{c}},
\end{aligned}
\end{equation*}
and then $m_{c}<m_{\bar{c}}+m_{c-\bar{c}}$.

If $m_{\bar{c}}<0$ and $m_{c-\bar{c}}<0$, then we have $\bar{c}>c_{0}$ and $c-\bar{c}>c_{0}$. Therefore,
\begin{equation*}
\begin{aligned}
m_{c}\leqslant (\frac{c}{\bar{c}})^{\frac{2(p-2)}{4-N(p-2)}+1}m_{\bar{c}}<m_{\bar{c}},
\end{aligned}
\end{equation*}
and
\begin{equation*}
\begin{aligned}
m_{c}\leqslant (\frac{c}{c-\bar{c}})^{\frac{2(p-2)}{4-N(p-2)}+1}m_{c-\bar{c}}<m_{c-\bar{c}},
\end{aligned}
\end{equation*}
and then $m_{c}<m_{\bar{c}}+m_{c-\bar{c}}$.
\end{proof}

\begin{lemma}\label{Lemma3.2}
Let $p\in(2+\frac{4s}{N},2+\frac{4}{N})$, $c>c_{0}$ and $\{u_{n}\}\subset S_{c}$ be a minimizing sequence of $m_{c}<0$.
Then there exists $x_{n}\subset \mathbb{R}^{N}$ such that  $\bar{u}_{n}:=u_{n}(x+x_{n})$ convergences to $u\not \equiv0$ in $L_{loc}^{2}(\mathbb{R}^{N})$, and $\{u_{n}\}\subset S_{c}$ is also a minimizing sequence of $m_{c}<0$.
\end{lemma}
\begin{proof}
If $\lim\limits_{n\to\infty}\int_{\mathbb{R}^{N}}|u_{n}|^{2^{*}}\mathrm{d}x=0$, then
\begin{equation*}
\begin{aligned}
0>m_{c}+o_{n}(1)=J(u_{n})=
\frac{1}{2}\|u_{n}\|_{D^{1,2}(\mathbb{R}^{N})}^{2}
+\frac{1}{2}\|u_{n}\|_{D^{s,2}(\mathbb{R}^{N})}^{2}\geqslant 0.
\end{aligned}
\end{equation*}
This is a contradiction.
We get that $\lim\limits_{n\to\infty}\int_{\mathbb{R}^{N}}|u_{n}|^{2^{*}}\mathrm{d}x>0$.
By using Lions' vanishing Lemma \cite{Willem1996Book}, there exists $x_{n}\subset \mathbb{R}^{N}$ such that  $\bar{u}_{n}:=u_{n}(x+x_{n})\rightharpoonup u\not \equiv0$ in $L^{2}_{loc}(\mathbb{R}^{N})$.
\end{proof}

We are now in a position to prove Theorem \ref{Theorem1.1}.
\begin{proof}[Proof of Theorem \ref{Theorem1.1}]
{\bf Step 1.} In this step, we consider the case $c\in(0,c_{0})$.
Define
\begin{equation*}
\begin{aligned}
J_{c}(u)=
\frac{1}{2}\|u\|_{D^{1,2}(\mathbb{R}^{N})}^{2}
+\frac{c^{\frac{2(p-2)}{4-N(p-2)}(s-1)}}{2}\|u\|_{D^{s,2}(\mathbb{R}^{N})}^{2}
-
\frac{1}{p}
\int_{\mathbb{R}^{N}}
|u|^{p}
\mathrm{d}x.
\end{aligned}
\end{equation*}
Clearly, $J_{1}(u)=J(u)$, $J_{c_{0}}(u)<J_{c}(u)$, and
\begin{equation*}
\begin{aligned}
J(c^{a}u(c^{b}\cdot))=c^{\frac{2(p-2)}{4-N(p-2)}+1}J_{c}(u).
\end{aligned}
\end{equation*}
where $a=\frac{2}{4-N(p-2)}$ and $b=\frac{p-2}{4-N(p-2)}$ in Lemma \ref{Lemma3.1}.

Assume on the contrary that $m_{c}=0$ is achievd, and $u_{c}$ is the minimizer. Then
\begin{equation*}
\begin{aligned}
0
=m_{c}
=J_{1}(u_{c}).
\end{aligned}
\end{equation*}
Set $\tilde{u}_{c}:=c^{-a}u_{c}(c^{-b}\cdot)$. Then $\tilde{u}_{c}\in S_{1}$ and
$
u_{c}=c^{a}\tilde{u}_{c}(c^{b}\cdot)$
and
\begin{equation*}
\begin{aligned}
0=&J(u_{c})\\
=&J(c^{a}\tilde{u}_{c}(c^{b}\cdot))\\
=&c^{\frac{2(p-2)}{4-N(p-2)}+1}J_{c}(\tilde{u}_{c}).
\end{aligned}
\end{equation*}
This shows that $J_{c}(\tilde{u}_{c})=0$ and $J_{c_{0}}(\tilde{u}_{c})<J_{c}(\tilde{u}_{c})=0$.

Set $\tilde{\tilde{u}}_{c_{0}}:=c_{0}^{a}\tilde{u}_{c}(c_{0}^{b}\cdot)$. Then $\tilde{\tilde{u}}_{c_{0}}\in S_{c_{0}}$ and
\begin{equation*}
\begin{aligned}
m_{c_{0}}\leqslant&J(\tilde{\tilde{u}}_{c_{0}})\\
=&J(c_{0}^{a}\tilde{u}_{c}(c_{0}^{b}\cdot))\\
=&c_{0}^{\frac{2(p-2)}{4-N(p-2)}+1}J_{c_{0}}(\tilde{u}_{c})<0.
\end{aligned}
\end{equation*}
This is a contradiction with $m_{c_{0}}=0$. Hence, the infimum is not achievd.

{\bf Step 2.}
In this step, we consider the case $c\in(c_{0},+\infty)$.
It follows from Lemma \ref{Lemma3.2} that
\begin{equation*}
\begin{aligned}
\bar{u}_{n}\rightharpoonup u\not\equiv0~\mathrm{in}~H^{1}(\mathbb{R}^{N}),
~~
\bar{u}_{n}\to u~\mathrm{a.e.}~\mathrm{in}~\mathbb{R}^{N}.
\end{aligned}
\end{equation*}
We claim that $\int_{\mathbb{R}^{N}}|u|^{2}\mathrm{d}x=c$.
Suppose on the contrary that
\begin{equation*}
\begin{aligned}
\int_{\mathbb{R}^{N}}
|u|^{2}
\mathrm{d}x
=\bar{c}
<c
=
\lim_{n\to\infty}
\int_{\mathbb{R}^{N}}
|\bar{u}_{n}|^{2}
\mathrm{d}x.
\end{aligned}
\end{equation*}
If $J(u)=m_{\bar{c}}$,
then by using Br\'{e}zis-Lieb's Lemma \cite{Brezis-Lieb1983PAMS} and Lemma \ref{Lemma3.1},
we know that
\begin{equation*}
\begin{aligned}
m_{c}
=&\lim_{n\to\infty}J(\bar{u}_{n})\\
=&J(u)
+\lim_{n\to\infty}J(\bar{u}_{n}-u)\\
\geqslant&m_{\bar{c}}+m_{c-\bar{c}}\\
>&m_{c}.
\end{aligned}
\end{equation*}
This is a contradiction.
Simiarly, if $J(u)>m_{\bar{c}}$, then we have that
\begin{equation*}
\begin{aligned}
m_{c}
=&J(u)
+\lim_{n\to\infty}J(\bar{u}_{n}-u)\\
>&m_{\bar{c}}+m_{c-\bar{c}}\\
\geqslant&m_{c}.
\end{aligned}
\end{equation*}
This is also a contradiction. Therefore, $\int_{\mathbb{R}^{3}}|u|^{2}\mathrm{d}x=c$. Moreover, we can see that
\begin{equation*}
\begin{aligned}
J(u)=m_{c},
\end{aligned}
\end{equation*}
and
\begin{equation*}
\begin{aligned}
\bar{u}_{n}\to u~\mathrm{in}~H^{1}(\mathbb{R}^{N}).
\end{aligned}
\end{equation*}

{\bf Step 3.}
In this step,
we consider the critical mass case $c=c_{0}$. There are two major difficulties to show the  relative compactness minimizing sequences (up to translation).
First,
we cannot rule out the vanishing of the minimizing sequences under $m_{c_{0}}=0$.
Second,
we can not rule out the dichotomy without the strict sub-additivity inequality $m_{c}<m_{\bar{c}}+m_{c-\bar{c}}$.

To overcome these challenges, we will use some arguments from \cite[Lemma 3.2]{LuoTJ-Hajaiej2022ANS}  and \cite[Page 1936, Line 12]{Catto-Dolbeault-Sanchez-Soler2013MMMAS}.
Let $c_{n}:=c_{0}+\frac{1}{n}$, for all $n\in \mathbb{N}^{+}$.
From $c_{n}>c_{0}$ and $m_{c_{n}}<0$,
we know by Step 2 that $m_{c_{n}}$ admits a minimizer $u_{c_{n}}\in S_{c_{n}}$ for all $n\in \mathbb{N}^{+}$.
Since $m_{c_{n}}\to m_{c_{0}}$,
it can be deduced that
\begin{equation*}
\begin{aligned}
\lim_{n\to\infty}
J(u_{c_{n}})
=
\lim_{n\to\infty}
m_{c_{n}}
=
m_{c_{0}}
=0.
\end{aligned}
\end{equation*}
It follows from \eqref{2.1} that
\begin{equation*}
\begin{aligned}
J(u_{c_{n}})
\geqslant&
\frac{1}{2}
\|u_{c_{n}}\|_{D^{1,2}(\mathbb{R}^{N})}^{2}
-
C
\|u_{c_{n}}\|_{D^{1,2}(\mathbb{R}^{N})}^{\frac{N(p-2)}{2}}
c_{n}^{\frac{p}{2}-\frac{N(p-2)}{4}}.
\end{aligned}
\end{equation*}
This implies that $\|u_{c_{n}}\|_{D^{1,2}(\mathbb{R}^{N})}$ is bounded via $p<2+\frac{4}{N}$.
Additionally,
$\|u_{c_{n}}\|_{L^{2}(\mathbb{R}^{N})}^{2}=c_{n}\to c_{0}>0$,
then we conclude that $\{u_{c_{n}}\}$ is bounded in $H^{1}(\mathbb{R}^{N})$.

We now prove that
\begin{equation*}
\begin{aligned}
\lim_{n\to\infty}
\int_{\mathbb{R}^{N}}
|u_{c_{n}}|^{p}
\mathrm{d}x
>0.
\end{aligned}
\end{equation*}
Indeed,
if $\lim\limits_{n\to\infty}
\int_{\mathbb{R}^{N}}
|u_{c_{n}}|^{p}
\mathrm{d}x=0$,
then by $\lim\limits_{n\to\infty}J(u_{c_{n}})$,
we obtain that
\begin{equation}\label{3.1}
\begin{aligned}
\lim_{n\to\infty}
\|u_{c_{n}}\|_{D^{1,2}(\mathbb{R}^{N})}^{2}
=
\lim_{n\to\infty}
\|u_{c_{n}}\|_{D^{s,2}(\mathbb{R}^{N})}^{2}
=0.
\end{aligned}
\end{equation}
On the other hand, by using inequality \eqref{2.2},
we have that
\begin{equation}\label{3.2}
\begin{aligned}
J(u_{c_{n}})
\geqslant&
\frac{1}{2}
\|u_{c_{n}}\|_{D^{1,2}(\mathbb{R}^{N})}^{2}
+
\frac{1}{2}
\|u_{c_{n}}\|_{D^{s,2}(\mathbb{R}^{N})}^{2}
-
C
\|u_{c_{n}}\|_{D^{s,2}(\mathbb{R}^{N})}^{\frac{N(p-2)}{2s}}
c_{n}^{\frac{p}{2}-\frac{N(p-2)}{4s}}\\
\geqslant&
\|u_{c_{n}}\|_{D^{s,2}(\mathbb{R}^{N})}^{2}
\left[
\frac{1}{2}
-
C
\|u_{c_{n}}\|_{D^{s,2}(\mathbb{R}^{N})}^{\frac{N(p-2)}{2s}-2}
c_{n}^{\frac{p}{2}-\frac{N(p-2)}{4s}}
\right],
\end{aligned}
\end{equation}
where we used the fact that $p>2+\frac{4s}{N}$.
By \eqref{3.1} and \eqref{3.2},
we have that $J(u_{c_{n}})\geqslant 0$ for $n$ large enough.
This contradicts the fact that $J(u_{c_{n}})=m_{c_{n}}<0$.
Then the claim is verified.
By applying Lions' vanishing Lemma \cite{Willem1996Book},
there exists $x_{n}\subset \mathbb{R}^{N}$ such that  $\bar{u}_{c_{n}}:=u_{c_{n}}(x+x_{n})$ convergences to $\bar{u}_{c_{0}}\not \equiv0$ in $L_{loc}^{2}(\mathbb{R}^{N})$.
By using the Br\'{e}zis-Lieb Lemma \cite{Brezis-Lieb1983PAMS},
we get that
\begin{equation}\label{3.3}
\begin{aligned}
0=
m_{c_{0}}
=
\lim_{n\to\infty}
m_{c_{n}}
=
\lim_{n\to\infty}
J(\bar{u}_{c_{n}})
=J(\bar{u}_{c_{0}})
+\lim_{n\to\infty}J(\bar{u}_{c_{n}}-\bar{u}_{c_{0}}).
\end{aligned}
\end{equation}
Let us prove that $\bar{u}_{c_{0}}$ is a minimizer of $m_{c_{0}}$.
Clearly,
\begin{equation*}
\begin{aligned}
c_{0}
=
\lim_{n\to\infty}
c_{n}
=
\lim_{n\to\infty}
\|\bar{u}_{c_{n}}\|_{L^{2}(\mathbb{R}^{N})}^{2}
=
\|\bar{u}_{c_{0}}\|_{L^{2}(\mathbb{R}^{N})}^{2}
+
\lim_{n\to\infty}
\|\bar{u}_{c_{n}}-\bar{u}_{c_{0}}\|_{L^{2}(\mathbb{R}^{N})}^{2}.
\end{aligned}
\end{equation*}
Since $0<\|\bar{u}_{c_{0}}\|_{L^{2}(\mathbb{R}^{N})}^{2}\leqslant c_{0}$,
we obtain that $J(\bar{u}_{c_{0}})\geqslant 0$
and
\begin{equation*}
\begin{aligned}
\lim_{n\to\infty}
\|\bar{u}_{c_{n}}-\bar{u}_{c_{0}}\|_{L^{2}(\mathbb{R}^{N})}^{2}
=
c_{0}
-
\|\bar{u}_{c_{0}}\|_{L^{2}(\mathbb{R}^{N})}^{2}<c_{0}.
\end{aligned}
\end{equation*}
This shows that
$\lim_{n\to\infty}J(\bar{u}_{c_{n}}-\bar{u}_{c_{0}})\geqslant 0$.
From \eqref{3.3},
we know that $J(\bar{u}_{c_{0}})=0$.
If $\|\bar{u}_{c_{0}}\|_{L^{2}(\mathbb{R}^{N})}^{2}<c_{0}$,
then $\bar{u}_{c_{0}}$ is a minimizer of $m_{c}$ for some $c<c_{0}$,
and we reach a contradiction with Step 1.
Hence, we know that $\|\bar{u}_{c_{0}}\|_{L^{2}(\mathbb{R}^{N})}^{2}=c_{0}$,
which shows that $\bar{u}_{c_{0}}$ is a minimizer of $m_{c_{0}}$.

{\bf Step 4.} For $c\geqslant c_{0}$, we know that $u\in S_{c}$ is a minimizer of $m_{c}$, then $u$ is a critical point of $J$ on the shpere $S_{c}$. There exists a Lagrange multiplier $\omega$ such that $(u,\omega)\in H^{1}(\mathbb{R}^{N})\times \mathbb{R}$ is a couple of weak solution to equation \eqref{1.1}.
Then
\begin{equation*}
\begin{aligned}
\omega\ c=\omega \int_{\mathbb{R}^{N}}|u|^{2}\mathrm{d}x
=&\int_{\mathbb{R}^{N}}|u|^{p}\mathrm{d}x
-\|u\|_{D^{1,2}(\mathbb{R}^{N})}^{2}
-\|u\|_{D^{s,2}(\mathbb{R}^{N})}^{2}\\
=&\frac{p-2}{p}
\int_{\mathbb{R}^{N}}|u|^{p}\mathrm{d}x-2J(u)\\
=&\frac{p-2}{p}
\int_{\mathbb{R}^{N}}|u|^{p}\mathrm{d}x-2m_{c}\\
\geqslant&\frac{p-2}{p}
\int_{\mathbb{R}^{N}}|u|^{p}\mathrm{d}x,
\end{aligned}
\end{equation*}
which shows that
\begin{equation*}
\begin{aligned}
w\geqslant \frac{p-2}{pc}
\int_{\mathbb{R}^{N}}|u|^{p}\mathrm{d}x>0.
\end{aligned}
\end{equation*}
This ends the proof.
\end{proof}

\section{Proof of Theorem \ref{Theorem1.2}}
In this section,
we prove Theorem \ref{Theorem1.2},
i.e.,
the existence of an optimizer of the Gagliardo-Nirenberg inequality for the mixed local-nonlocal operator.

\begin{lemma}\label{Lemma4.1}
For $p\in(2+\frac{4s}{N},2+\frac{4}{N})$ and $u\in H^{1}(\mathbb{R}^{N})$,
there exists a constant $C_{1,s,p}>0$ such that
\begin{equation}\label{4.1}
\begin{aligned}
\|u\|_{L^{p}(\mathbb{R}^{N})}^{p}
\leqslant C_{1,s,p}
\|u\|_{D^{s,2}(\mathbb{R}^{N})}^{\frac{2N+4-pN}{2(1-s)}}
\|u\|_{D^{1,2}(\mathbb{R}^{N})}^{\frac{pN-2N-4s}{2(1-s)}}
\|u\|_{L^{2}(\mathbb{R}^{N})}^{p-2}.
\end{aligned}
\end{equation}
\end{lemma}
\begin{proof}
Taking $p=2+\frac{4}{N}$ and  $q=2+\frac{4s}{N}$ in \eqref{2.1} and \eqref{2.2},
we have that
\begin{equation}\label{4.2}
\begin{aligned}
\|u\|_{L^{2+\frac{4}{N}}(\mathbb{R}^{N})}^{2+\frac{4}{N}}
\leqslant
C
\|u\|_{D^{1,2}(\mathbb{R}^{N})}^{2}
\|u\|_{L^{2}(\mathbb{R}^{N})}^{\frac{4}{N}},
\end{aligned}
\end{equation}
and
\begin{equation}\label{4.3}
\begin{aligned}
\|u\|_{L^{2+\frac{4s}{N}}(\mathbb{R}^{N})}^{2+\frac{4s}{N}}
\leqslant&
C
\|u\|_{D^{s,2}(\mathbb{R}^{N})}^{2}
\|u\|_{L^{2}(\mathbb{R}^{N})}^{\frac{4s}{N}}.
\end{aligned}
\end{equation}
Applying the H\"{o}lder inequality for $p\in(2+\frac{4s}{N},2+\frac{4}{N})$,
we obtain that
\begin{equation*}
\begin{aligned}
\int_{\mathbb{R}^{N}}|u|^{p}\mathrm{d}x
\leqslant\left(\int_{\mathbb{R}^{N}}|u|^{2+\frac{4s}{N}}\mathrm{d}x\right)^{\frac{(2+\frac{4}{N})-p}{\frac{4}{N}(1-s)}}
\left(\int_{\mathbb{R}^{N}}|u|^{2+\frac{4}{N}}\mathrm{d}x\right)^{\frac{p-(2+\frac{4s}{N})}{\frac{4}{N}(1-s)}},
\end{aligned}
\end{equation*}
and then using \eqref{4.2} and \eqref{4.3},
\begin{equation*}
\begin{aligned}
&\int_{\mathbb{R}^{N}}|u|^{p}\mathrm{d}x\\
\leqslant&C
\left(\|u\|_{D^{s,2}(\mathbb{R}^{N})}^{2}
\|u\|_{L^{2}(\mathbb{R}^{N})}^{\frac{4s}{N}}
\right)^{\frac{(2+\frac{4}{N})-p}{\frac{4}{N}(1-s)}}
\left(\|u\|_{D^{1,2}(\mathbb{R}^{N})}^{2}
\|u\|_{L^{2}(\mathbb{R}^{N})}^{\frac{4}{N}}\right)^{\frac{p-(2+\frac{4s}{N})}{\frac{4}{N}(1-s)}}\\
=&C
\|u\|_{D^{s,2}(\mathbb{R}^{N})}^{\frac{2N+4-pN}{2(1-s)}}
\|u\|_{D^{1,2}(\mathbb{R}^{N})}^{\frac{pN-2N-4s}{2(1-s)}}
\|u\|_{L^{2}(\mathbb{R}^{N})}^{p-2}.
\end{aligned}
\end{equation*}
\end{proof}

\begin{lemma}\label{Lemma4.2}
Let $N\geqslant 3$, $s\in(0,1)$ and $p \in(2+\frac{4s}{N},2+\frac{4}{N})$.
Then there exists an optimizer $Q\in H^{1}(\mathbb{R}^{N})$ of inequality \eqref{4.1},
that is a weak solution of
\begin{equation}\label{4.4}
\begin{aligned}
-\Delta Q
+(-\Delta)^{s}Q+\omega Q=|Q|^{p-2}Q, \ \ x\in \mathbb{R}^{N},
\end{aligned}
\end{equation}
where $\omega>0,$ and
\begin{equation*}
\begin{aligned}
C_{1,s,p}
=
\frac
{\|Q\|_{L^{p}(\mathbb{R}^{N})}^{p}}
{
\|Q\|_{D^{s,2}(\mathbb{R}^{N})}^{\frac{2N+4-pN}{2(1-s)}}
\|Q\|_{D^{1,2}(\mathbb{R}^{N})}^{\frac{pN-2N-4s}{2(1-s)}}
\|Q\|_{L^{2}(\mathbb{R}^{N})}^{p-2}
}.
\end{aligned}
\end{equation*}
\end{lemma}
\begin{proof}
{\bf Step 1.}
In this step, we study the existence of an optimizer for the best constant $C_{1,s,p}$.
We define the Weinstein functional
\begin{equation*}
\begin{aligned}
W_{p}(u)=
\frac{
\|u\|_{D^{s,2}(\mathbb{R}^{N})}^{\frac{2N+4-pN}{2(1-s)}}
\|u\|_{D^{1,2}(\mathbb{R}^{N})}^{\frac{pN-2N-4s}{2(1-s)}}
\|u\|_{L^{2}(\mathbb{R}^{N})}^{p-2}
}
{\|u\|_{L^{p}(\mathbb{R}^{N})}^{p}},
\end{aligned}
\end{equation*}
so that
\begin{equation*}
\begin{aligned}
C_{1,s,p}^{-1}=\inf_{u\in H^{1}(\mathbb{R}^{N})\backslash \{0\}}W_{p}(u).
\end{aligned}
\end{equation*}
From \eqref{4.1}, we know that there exists a minimizing sequence $\{u_{n}\}\in H^{1}(\mathbb{R}^{N})$ such that
\begin{equation*}
\begin{aligned}
\lim_{n\to\infty}
W_{p}(u_{n})
=C_{1,s,p}^{-1}.
\end{aligned}
\end{equation*}
It is straightforward that $\{v_{n}:=|u_{n}|\}$ is also a minimizing sequence.
Denote by $v_{n}^{*}$ the symmetric decreasing rearrangement of $v_{n}$.
By the rearrangement inequalities \cite[Theorems 3.4 and 3.7]{Lieb2001} and \cite{Burchard-Hajaiej2006JFA,Hajaiej-Stuart2003PLMS,Hajaiej2005Edinburgh},
we have that:
\begin{equation}\label{4.5}
\begin{aligned}
\int_{\mathbb{R}^{N}}
|\nabla v_{n}^{*}|^{2}
\mathrm{d}x
\leqslant
\int_{\mathbb{R}^{N}}
|\nabla v_{n}|^{2}
\mathrm{d}x,
\end{aligned}
\end{equation}
and
\begin{equation}\label{4.6}
\begin{aligned}
\int_{\mathbb{R}^{N}}
|v_{n}^{*}|^{p}
\mathrm{d}x
=
\int_{\mathbb{R}^{N}}
|v_{n}|^{p}
\mathrm{d}x,
\end{aligned}
\end{equation}
and
\begin{equation}\label{4.7}
\begin{aligned}
\int_{\mathbb{R}^{N}}
\int_{\mathbb{R}^{N}}
\frac{|v_{n}^{*}(x)-v_{n}^{*}(y)|^{2}}{|x-y|^{N+2s}}
\mathrm{d}x
\mathrm{d}y
\leqslant
\int_{\mathbb{R}^{N}}
\int_{\mathbb{R}^{N}}
\frac{|v_{n}(x)-v_{n}(y)|^{2}}{|x-y|^{N+2s}}
\mathrm{d}x
\mathrm{d}y.
\end{aligned}
\end{equation}
Combining \eqref{4.5}-\eqref{4.7},
we get that $\{v_{n}^{*}\}$
is also a minimizing sequence.

We now rescale the sequence $\{v_{n}^{*}\}$ by setting
\begin{equation*}
\begin{aligned}
w_{n}(x)=\lambda_{1,n}v_{n}^{*}(\lambda_{2,n}x),
\end{aligned}
\end{equation*}
where
\begin{equation*}
\begin{aligned}
\lambda_{1,n}
=\frac{\left[
\int_{\mathbb{R}^{N}}|v_{n}^{*}|^{2}\mathrm{d}x
\right]^{\frac{N-2}{4}}}{\left[
\int_{\mathbb{R}^{N}}|\nabla v_{n}^{*}|\mathrm{d}x
\right]^{\frac{N}{4}}}
~~\mathrm{and}~~
\lambda_{2,n}
=
\left[
\frac{
\int_{\mathbb{R}^{N}}|v_{n}^{*}|^{2}\mathrm{d}x
}
{\int_{\mathbb{R}^{N}}|\nabla v_{n}^{*}|\mathrm{d}x}
\right]^{\frac{1}{2}}.
\end{aligned}
\end{equation*}
Then

\begin{equation*}
\begin{aligned}
\int_{\mathbb{R}^{N}}|\nabla w_{n}|^{2}\mathrm{d}x
=
\lambda_{1,n}^{2}
\lambda_{2,n}^{2-N}
\int_{\mathbb{R}^{N}}|\nabla v_{n}^{*}|^{2}\mathrm{d}x=1,
\end{aligned}
\end{equation*}
and
\begin{equation*}
\begin{aligned}
\int_{\mathbb{R}^{N}}|w_{n}|^{2}\mathrm{d}x
=
\lambda_{1,n}^{2}\lambda_{2,n}^{-N}
\int_{\mathbb{R}^{N}}|v_{n}^{*}|^{2}\mathrm{d}x=1,
\end{aligned}
\end{equation*}
and
\begin{equation*}
\begin{aligned}
\int_{\mathbb{R}^{N}}\int_{\mathbb{R}^{N}}\frac{|w_{n}(x)-w_{n}(y)|^{2}}{|x-y|^{N+2s}}\mathrm{d}x\mathrm{d}y
=&
\lambda_{1,n}^{2}
\lambda_{2,n}^{2s-N}
\int_{\mathbb{R}^{N}}\int_{\mathbb{R}^{N}}\frac{|v_{n}^{*}(x)-v_{n}^{*}(y)|^{2}}{|x-y|^{N+2s}}\mathrm{d}x\mathrm{d}y,
\end{aligned}
\end{equation*}
and
\begin{equation*}
\begin{aligned}
\int_{\mathbb{R}^{N}}|w_{n}|^{p}\mathrm{d}x
=\lambda_{1,n}^{p}\lambda_{2,n}^{-N}
\int_{\mathbb{R}^{N}}|v_{n}^{*}|^{p}\mathrm{d}x.
\end{aligned}
\end{equation*}
The Weinstein functional $W_{p}(\cdot)$ is invariant under rescale
\begin{equation*}
\begin{aligned}
W_{p}(w_{n})
=&\frac{
\|w_{n}\|_{D^{s,2}(\mathbb{R}^{N})}^{\frac{2N+4-pN}{2(1-s)}}
\|w_{n}\|_{D^{1,2}(\mathbb{R}^{N})}^{\frac{pN-2N-4s}{2(1-s)}}
\|w_{n}\|_{L^{2}(\mathbb{R}^{N})}^{p-2}
}
{\|w_{n}\|_{L^{p}(\mathbb{R}^{N})}^{p}}\\
=&
\frac{
[\lambda_{1,n}^{2}
\lambda_{2,n}^{2s-N}]^{\frac{2N+4-pN}{4(1-s)}}
[\lambda_{1,n}^{2}
\lambda_{2,n}^{2-N}]^{\frac{pN-2N-4s}{4(1-s)}}
[\lambda_{1,n}^{2}\lambda_{2,n}^{-N}]^{\frac{p-2}{2}}
}
{\lambda_{1,n}^{p}\lambda_{2,n}^{-N}}
W_{p}(v_{n}^{*})\\
=&W_{p}(v_{n}^{*}).
\end{aligned}
\end{equation*}
Clearly, $\{w_{n}\}$ is a radial and radially decreasing minimizing sequence which is bounded in  $H^{1}(\mathbb{R}^{N})$.

By virtue of $\|w_{n}\|_{D^{1,2}(\mathbb{R}^{N})}^{2}=\|w_{n}\|_{L^{2}(\mathbb{R}^{N})}^{2}=1$, $0<C_{1,s,p}^{-1}<\infty$ and \eqref{2.2} that
\begin{equation}\label{Nonvanishing}
\begin{aligned}
C_{1,s,p}^{-1}
=&\lim_{n\to\infty}
\frac{
\|w_{n}\|_{D^{s,2}(\mathbb{R}^{N})}^{\frac{2N+4-pN}{2(1-s)}}
\|w_{n}\|_{D^{1,2}(\mathbb{R}^{N})}^{\frac{pN-2N-4s}{2(1-s)}}
\|w_{n}\|_{L^{2}(\mathbb{R}^{N})}^{p-2}
}
{\|w_{n}\|_{L^{p}(\mathbb{R}^{N})}^{p}}\\
=&
\lim_{n\to\infty}
\frac{
\|w_{n}\|_{D^{s,2}(\mathbb{R}^{N})}^{\frac{N(p-2)}{2s}\left[(\frac{4}{N(p-2)}-1)
\frac{s}{1-s}\right]}
\|w_{n}\|_{L^{2}(\mathbb{R}^{N})}^{\left[p-\frac{N(p-2)}{2s}\right]\left[(\frac{4}{N(p-2)}-1)
\frac{s}{1-s}\right]}
}
{\|w_{n}\|_{L^{p}(\mathbb{R}^{N})}^{p}}\\
\geqslant&C
\lim_{n\to\infty}
\frac{
\|w_{n}\|_{L^{p}(\mathbb{R}^{N})}^{p\left[(\frac{4}{N(p-2)}-1)
\frac{s}{1-s}\right]}}
{\|w_{n}\|_{L^{p}(\mathbb{R}^{N})}^{p}}\\
=&C\lim_{n\to\infty}\|w_{n}\|_{L^{p}(\mathbb{R}^{N})}^{p\left[(\frac{4}{N(p-2)}-1)
\frac{s}{1-s}-1\right]},
\end{aligned}
\tag{$Nonvanishing$}
\end{equation}
which shows, by $0<(\frac{4}{N(p-2)}-1), that
\frac{s}{1-s}<1$,
\begin{equation}\label{4.8}
\begin{aligned}
0<[CC_{1,s,p}]^{\frac{1}{1-(\frac{4}{N(p-2)}-1)
\frac{s}{1-s}}}\leqslant\lim\limits_{n\to\infty}\|w_{n}\|_{L^{p}(\mathbb{R}^{N})}^{p}.
\end{aligned}
\end{equation}
From the compact embedding $H^{1}_{rad}(\mathbb{R}^{N})\hookrightarrow\hookrightarrow L^{q}(\mathbb{R}^{N})$,
$q\in(2,2^{*})$ \cite{Willem1996Book},
and \eqref{4.8},
we deduce that
\begin{equation*}
\begin{aligned}
\lim\limits_{n\to\infty}\|w_{n}\|_{L^{p}(\mathbb{R}^{N})}^{p}
= \|w\|_{L^{p}(\mathbb{R}^{N})}^{p}\not=0.
\end{aligned}
\end{equation*}
It follows from the Br\'{e}zis-Lieb Lemma \cite{Brezis-Lieb1983PAMS} and $\|w\|_{L^{p}(\mathbb{R}^{N})}^{p}\not=0$ that
\begin{equation*}
\begin{aligned}
C_{1,s,p}^{-1}
=&\lim_{n\to\infty}
\frac
{
\|w_{n}\|_{D^{s,2}(\mathbb{R}^{N})}^{\frac{2N+4-pN}{2(1-s)}}
\|w_{n}\|_{D^{1,2}(\mathbb{R}^{N})}^{\frac{pN-2N-4s}{2(1-s)}}
\|w_{n}\|_{L^{2}(\mathbb{R}^{N})}^{p-2}
}
{\|w_{n}\|_{L^{p}(\mathbb{R}^{N})}^{p}}\\
\geqslant&
\frac
{
\|w\|_{D^{s,2}(\mathbb{R}^{N})}^{\frac{2N+4-pN}{2(1-s)}}
\|w\|_{D^{1,2}(\mathbb{R}^{N})}^{\frac{pN-2N-4s}{2(1-s)}}
\|w\|_{L^{2}(\mathbb{R}^{N})}^{p-2}
}
{\|w\|_{L^{p}(\mathbb{R}^{N})}^{p}}\\
\geqslant&
C_{1,s,p}^{-1}.
\end{aligned}
\end{equation*}
Thus,
\begin{equation*}
\begin{aligned}
\|w\|_{D^{1,2}(\mathbb{R}^{N})}^{2}=1,~ \int_{\mathbb{R}^{2}}|w|^{2}\mathrm{d}x=1,~
W_{p}(w)=C_{1,s,p}^{-1}.
\end{aligned}
\end{equation*}
Therefore, $w$ is a radial non-negative minimizer for the Weinstein functional $W_{p}$.

{\bf Step 2.} Note that $w$ is a minimizer of $W_{p}$. Then it satisfies
\begin{equation*}
\begin{aligned}
\left.\frac{\mathrm{d}}{\mathrm{d}\varepsilon}\right|_{\varepsilon=0}W_{p}(w+\varepsilon \phi)=0,
\end{aligned}
\end{equation*}
for all $\phi\in C^{\infty}_{0}(\mathbb{R}^{N})$. Thus, $w$ is a weak solution to the following equation
\begin{equation*}
\begin{aligned}
&-\Delta w
+\frac{2N+4-pN}{pN-2N-4s}
(-\Delta)^{s}w
\frac{1}{\|w\|_{D^{s,2}(\mathbb{R}^{N})}^{2}}
+\frac{2(p-2)(1-s)}{pN-2N-4s}w\\
=&\frac{2p(1-s)}{pN-2N-4s}|w|^{p-2}w
\frac{1}
{\|w\|_{L^{p}(\mathbb{R}^{N})}^{p}}.
\end{aligned}
\end{equation*}
We define
\begin{equation}\label{4.9}
\begin{aligned}
Q(x)=\lambda w(\mu x),
\end{aligned}
\end{equation}
where
\begin{equation*}
\begin{aligned}
\begin{cases}
\mu=\left[\frac{2N+4-pN}{pN-2N-4s}
\|w\|_{D^{s,2}(\mathbb{R}^{N})}^{-2}
\right]^{\frac{1}{2s-2}},\\
\lambda
=
\left[\frac{2p(1-s)}{pN-2N-4s}
\|w\|_{L^{p}(\mathbb{R}^{N})}^{-p}\right]^{\frac{1}{p-2}}
\left[\frac{2N+4-pN}{pN-2N-4s}
\|w\|_{D^{s,2}(\mathbb{R}^{N})}^{-2}
\right]^{^{\frac{2}{(2s-2)(p-2)}}}.
\end{cases}
\end{aligned}
\end{equation*}
Then $Q$ satisfies
\begin{equation*}
\begin{aligned}
-\Delta Q
+
(-\Delta)^{s}Q
+\omega Q
=|Q|^{p-2}Q,
\end{aligned}
\end{equation*}
with
\begin{equation*}
\begin{aligned}
\omega
=\frac{pN-2N-4s}{2(p-2)(1-s)}
\left[
\frac{2N+4-pN}{pN-2N-4s}
\|w\|_{D^{s,2}(\mathbb{R}^{N})}^{-2}
\right]^{\frac{-1}{s-1}}.
\end{aligned}
\end{equation*}
Particularly, we know that $Q$ is non-negative radial and radially symmetric.
\end{proof}

\begin{remark}
In the Weinstein functional,
if the numerator has two terms,
then one can get the non-vanishing of the minimizing sequence in a straightforward way \cite{Weinstein1983CMP}.
However,
in our case there are three terms in the numerator of our $W_{p}(\cdot)$.
It is hard to rule out the vanishing of minimizing sequence.
We overcome this challenge using an innovative and novel idea,
see \eqref{Nonvanishing}.
\end{remark}

\section{Proof of Theorem \ref{Theorem1.3}}
In this section,
we present the proof of Theorem \ref{Theorem1.3}.
\begin{lemma}\label{Lemma5.1}
Let $N\geqslant 3$, $s\in(0,1)$ and $p \in(2+\frac{4s}{N},2+\frac{4}{N})$.
Set
\begin{equation*}
\begin{aligned}
\beta(Q)=\frac{\int_{\mathbb{R}^{N}}\int_{\mathbb{R}^{N}}\frac{|Q(x)-Q(y)|^{2}}{|x-y|^{N+2s}}\mathrm{d}x\mathrm{d}y}
{\int_{\mathbb{R}^{N}}|\nabla Q|^{2}\mathrm{d}x},
\end{aligned}
\end{equation*}
where $Q$ is defined in Lemma \ref{Lemma4.2}.
Then
\begin{equation*}
\begin{aligned}
\beta(Q)=\frac{2N+4-pN}{pN-2N-4s}.
\end{aligned}
\end{equation*}
\end{lemma}

\begin{proof}
Using \eqref{4.9}, we compute
\begin{equation*}
\begin{aligned}
\beta(Q)
=&\frac{\int_{\mathbb{R}^{N}}\int_{\mathbb{R}^{N}}\frac{|Q(x)-Q(y)|^{2}}{|x-y|^{N+2s}}\mathrm{d}x\mathrm{d}y}
{\int_{\mathbb{R}^{N}}|\nabla Q|^{2}\mathrm{d}x}\\
=&\frac{\lambda^{2}
\mu^{2s-N}\int_{\mathbb{R}^{N}}\int_{\mathbb{R}^{N}}\frac{|w(x)-w(y)|^{2}}{|x-y|^{N+2s}}\mathrm{d}x\mathrm{d}y}
{\lambda^{2}
\mu^{2-N} \int_{\mathbb{R}^{N}}|\nabla w|^{2}\mathrm{d}x}\\
=&
\mu^{2s-2}
\frac{\int_{\mathbb{R}^{N}}\int_{\mathbb{R}^{N}}\frac{|w(x)-w(y)|^{2}}{|x-y|^{N+2s}}\mathrm{d}x\mathrm{d}y}
{\int_{\mathbb{R}^{N}}|\nabla w|^{2}\mathrm{d}x}.
\end{aligned}
\end{equation*}
Setting $\int_{\mathbb{R}^{N}}|\nabla w|\mathrm{d}x=1$ and $\mu=\left[\frac{2N+4-pN}{pN-2N-4s}
\|w\|_{D^{s,2}(\mathbb{R}^{N})}^{-2}
\right]^{\frac{1}{2s-2}}$ into the above equation, we infer that
\begin{equation*}
\begin{aligned}
\beta(Q)
=&
\mu^{2s-2}
\|w\|_{D^{s,2}(\mathbb{R}^{N})}^{2}\\
=&
\left[\frac{2N+4-pN}{pN-2N-4s}
\|w\|_{D^{s,2}(\mathbb{R}^{N})}^{-2}
\right]
\|w\|_{D^{s,2}(\mathbb{R}^{N})}^{2}\\
=&\frac{2N+4-pN}{pN-2N-4s}.
\end{aligned}
\end{equation*}
\end{proof}

\begin{lemma}\label{Lemma5.2}
Let $N\geqslant 3$, $s\in(0,1)$ and $p \in(2+\frac{4s}{N},2+\frac{4}{N})$.
Set
\begin{equation*}
\begin{aligned}
\gamma(Q)
=&\frac{\int_{\mathbb{R}^{N}}\int_{\mathbb{R}^{N}}\frac{|Q(x)-Q(y)|^{2}}{|x-y|^{N+2s}}\mathrm{d}x\mathrm{d}y}
{\int_{\mathbb{R}^{N}}|Q|^{p}\mathrm{d}x},
\end{aligned}
\end{equation*}
where $Q$ is defined in Lemma \ref{Lemma4.2}.
Then
\begin{equation*}
\begin{aligned}
\gamma(Q)=\frac{2N+4-pN}{2p(1-s)}.
\end{aligned}
\end{equation*}
And,
for $s\in(0,1)$, we have
\begin{equation*}
\begin{aligned}
\|Q\|^{2}_{D^{1,2}(\mathbb{R}^{N})}
=\frac{pN-2N-4s}{2p(1-s)}\int_{\mathbb{R}^{N}}|Q|^{p}\mathrm{d}x,
\end{aligned}
\end{equation*}
and
\begin{equation*}
\begin{aligned}
\|Q\|^{2}_{D^{s,2}(\mathbb{R}^{N})}
=\frac{2N+4-pN}{2p(1-s)}
\int_{\mathbb{R}^{N}}|Q|^{p}\mathrm{d}x.
\end{aligned}
\end{equation*}
\end{lemma}

\begin{proof}
By virtue of \eqref{4.9} again, we deduce that
\begin{equation*}
\begin{aligned}
\gamma(Q)
=&\frac{\lambda^{2}
\mu^{2s-N}\int_{\mathbb{R}^{N}}\int_{\mathbb{R}^{N}}\frac{|w(x)-w(y)|^{2}}{|x-y|^{N+2s}}\mathrm{d}x\mathrm{d}y}
{\lambda^{p}
\mu^{-N}
\int_{\mathbb{R}^{N}}|w|^{p}\mathrm{d}x}\\
=&
\lambda^{2-p}
\mu^{2s}
\frac{\int_{\mathbb{R}^{N}}\int_{\mathbb{R}^{N}}\frac{|w(x)-w(y)|^{2}}{|x-y|^{N+2s}}\mathrm{d}x\mathrm{d}y}
{\int_{\mathbb{R}^{N}}|w|^{p}\mathrm{d}x}\\
=&
\left[\frac{2p(1-s)}{pN-2N-4s}
\|w\|_{L^{p}(\mathbb{R}^{N})}^{-p}\right]^{-1}
\left[\frac{2N+4-pN}{pN-2N-4s}
\|w\|_{D^{s,2}(\mathbb{R}^{N})}^{-2}
\right]^{^{\frac{-2}{2s-2}}}\\
&\times
\left[\frac{2N+4-pN}{pN-2N-4s}
\|w\|_{D^{s,2}(\mathbb{R}^{N})}^{-2}
\right]^{\frac{2s}{2s-2}}
\frac{\int_{\mathbb{R}^{N}}\int_{\mathbb{R}^{N}}\frac{|w(x)-w(y)|^{2}}{|x-y|^{N+2s}}\mathrm{d}x\mathrm{d}y}
{\int_{\mathbb{R}^{N}}|w|^{p}\mathrm{d}x}\\
=&
\frac{pN-2N-4s}{2p(1-s)}\frac{2N+4-pN}{pN-2N-4s}\\
=&\frac{2N+4-pN}{2p(1-s)}.
\end{aligned}
\end{equation*}
Combining $\beta(Q)$ and $\gamma(Q)$, we have the desired result.
\end{proof}

\begin{lemma}\label{Lemma5.3}
Let $N\geqslant 3$, $s\in(0,1)$ and $p \in(2+\frac{4s}{N},2+\frac{4}{N})$.
Then
\begin{equation*}
\begin{aligned}
\frac{1}{2}
\int_{\mathbb{R}^{N}}
|\nabla Q|^{2}
\mathrm{d}x
+\frac{1}{2}
\int_{\mathbb{R}^{N}}
\int_{\mathbb{R}^{N}}
\frac{|Q(x)-Q(y)|^{2}}{|x-y|^{N+2s}}
\mathrm{d}x
\mathrm{d}y
-\frac{1}{p}
\int_{\mathbb{R}^{N}}
|Q|^{p}
\mathrm{d}x=0,
\end{aligned}
\end{equation*}
and
\begin{equation*}
\begin{aligned}
\int_{\mathbb{R}^{N}}
|Q|^{2}
\mathrm{d}x\geqslant c_{0}.
\end{aligned}
\end{equation*}
\end{lemma}

\begin{proof}
We know that
$Q$ satisifies the following Nehari identity
\begin{equation*}
\begin{aligned}
\int_{\mathbb{R}^{N}}
|\nabla Q|^{2}
\mathrm{d}x
+
\int_{\mathbb{R}^{N}}
\int_{\mathbb{R}^{N}}
\frac{|Q(x)-Q(y)|^{2}}{|x-y|^{N+2s}}
\mathrm{d}x
\mathrm{d}y
+
\omega
\int_{\mathbb{R}^{N}}
|Q|^{2}
\mathrm{d}x
-
\int_{\mathbb{R}^{N}}
|Q|^{p}
\mathrm{d}x
=0,
\end{aligned}
\end{equation*}
and the Pohoz\'{a}ev identity
\begin{equation*}
\begin{aligned}
&\frac{N-2}{2}
\int_{\mathbb{R}^{N}}
|\nabla Q|^{2}
\mathrm{d}x
+\frac{N-2s}{2}
\int_{\mathbb{R}^{N}}
\int_{\mathbb{R}^{N}}
\frac{|Q(x)-Q(y)|^{2}}{|x-y|^{N+2s}}
\mathrm{d}x
\mathrm{d}y\\
&+
\frac{N}{2}
\omega
\int_{\mathbb{R}^{N}}
|Q|^{2}
\mathrm{d}x
-\frac{N}{p}
\int_{\mathbb{R}^{N}}
|Q|^{p}
\mathrm{d}x
=0.
\end{aligned}
\end{equation*}
Then
\begin{equation*}
\begin{aligned}
\int_{\mathbb{R}^{N}}
|\nabla Q|^{2}
\mathrm{d}x
+s
\int_{\mathbb{R}^{N}}
\int_{\mathbb{R}^{N}}
\frac{|Q(x)-Q(y)|^{2}}{|x-y|^{N+2s}}
\mathrm{d}x
\mathrm{d}y
-\frac{(p-2)N}{2p}
\int_{\mathbb{R}^{N}}
|Q|^{p}
\mathrm{d}x=0
\end{aligned}
\end{equation*}
From Lemmas \ref{Lemma5.1} and \ref{Lemma5.2},
we get that
\begin{equation*}
\begin{aligned}
&\frac{1}{2}
\int_{\mathbb{R}^{N}}
|\nabla Q|^{2}
\mathrm{d}x
+(s-\frac{1}{2})
\int_{\mathbb{R}^{N}}
\int_{\mathbb{R}^{N}}
\frac{|Q(x)-Q(y)|^{2}}{|x-y|^{N+2s}}
\mathrm{d}x
\mathrm{d}y
-\frac{1}{p}
(\frac{(p-2)N}{2}-1)
\int_{\mathbb{R}^{N}}
|Q|^{p}
\mathrm{d}x\\
=&
\left[\frac{1}{2}\frac{pN-2N-4s}{2N+4-pN}
+s-\frac{1}{2}
-\frac{1}{p}
(\frac{(p-2)N}{2}-1)
\frac{2p(1-s)}{2N+4-pN}\right]
\int_{\mathbb{R}^{N}}
\int_{\mathbb{R}^{N}}
\frac{|Q(x)-Q(y)|^{2}}{|x-y|^{N+2s}}
\mathrm{d}x
\mathrm{d}y\\
=&0.
\end{aligned}
\end{equation*}
Therefore,
\begin{equation*}
\begin{aligned}
0=&
\int_{\mathbb{R}^{N}}
|\nabla Q|^{2}
\mathrm{d}x
+s
\int_{\mathbb{R}^{N}}
\int_{\mathbb{R}^{N}}
\frac{|Q(x)-Q(y)|^{2}}{|x-y|^{N+2s}}
\mathrm{d}x
\mathrm{d}y
-\frac{(p-2)N}{2p}
\int_{\mathbb{R}^{N}}
|Q|^{p}
\mathrm{d}x\\
=&\frac{1}{2}
\int_{\mathbb{R}^{N}}
|\nabla Q|^{2}
\mathrm{d}x
+(s-\frac{1}{2})
\int_{\mathbb{R}^{N}}
\int_{\mathbb{R}^{N}}
\frac{|Q(x)-Q(y)|^{2}}{|x-y|^{N+2s}}
\mathrm{d}x
\mathrm{d}y
-\frac{1}{p}
(\frac{(p-2)N}{2}-1)
\int_{\mathbb{R}^{N}}
|Q|^{p}
\mathrm{d}x\\
&+\frac{1}{2}
\int_{\mathbb{R}^{N}}
|\nabla Q|^{2}
\mathrm{d}x
+\frac{1}{2}
\int_{\mathbb{R}^{N}}
\int_{\mathbb{R}^{N}}
\frac{|Q(x)-Q(y)|^{2}}{|x-y|^{N+2s}}
\mathrm{d}x
\mathrm{d}y
-\frac{1}{p}
\int_{\mathbb{R}^{N}}
|Q|^{p}
\mathrm{d}x\\
=&\frac{1}{2}
\int_{\mathbb{R}^{N}}
|\nabla Q|^{2}
\mathrm{d}x
+\frac{1}{2}
\int_{\mathbb{R}^{N}}
\int_{\mathbb{R}^{N}}
\frac{|Q(x)-Q(y)|^{2}}{|x-y|^{N+2s}}
\mathrm{d}x
\mathrm{d}y
-\frac{1}{p}
\int_{\mathbb{R}^{N}}
|Q|^{p}
\mathrm{d}x.
\end{aligned}
\end{equation*}
From Theorem \ref{Theorem1.1},
we know that
\begin{equation*}
\begin{aligned}
\int_{\mathbb{R}^{N}}
|Q|^{2}
\mathrm{d}x\geqslant c_{0}.
\end{aligned}
\end{equation*}
\end{proof}

\begin{proof}[Proof of Theorem \ref{Theorem1.3}]
By applying Lemmas \ref{Lemma5.1} and \ref{Lemma5.2}, one has
\begin{equation*}
\begin{aligned}
&C_{1,s,p}\\
=&
\left[
\frac{2N+4-pN}{2p(1-s)}
\int_{\mathbb{R}^{N}}|Q|^{p}\mathrm{d}x
\right]^{-\frac{2N+4-pN}{4(1-s)}}
\left[
\frac{pN-2N-4s}{2p(1-s)}\int_{\mathbb{R}^{N}}|Q|^{p}\mathrm{d}x
\right]^{-\frac{pN-2N-4s}{4(1-s)}}
\frac
{\|Q\|_{L^{p}(\mathbb{R}^{N})}^{p}}
{
\|Q\|_{L^{2}(\mathbb{R}^{N})}^{p-2}
}\\
=&\left[
\frac{2N+4-pN}{2p(1-s)}
\right]^{-\frac{2N+4-pN}{4(1-s)}}
\left[
\frac{pN-2N-4s}{2p(1-s)}
\right]^{-\frac{pN-2N-4s}{4(1-s)}}
\frac{1}{\|Q\|_{L^{2}(\mathbb{R}^{N})}^{p-2}}.
\end{aligned}
\end{equation*}
Set
\begin{equation*}
\begin{aligned}
C_{1,s,p}':=\left[
\frac{2N+4-pN}{2p(1-s)}
\right]^{-\frac{2N+4-pN}{4(1-s)}}
\left[
\frac{pN-2N-4s}{2p(1-s)}
\right]^{-\frac{pN-2N-4s}{4(1-s)}}
\frac{1}{c_{0}^{\frac{p-2}{2}}}.
\end{aligned}
\end{equation*}
We now show that $\|Q\|_{L^{2}(\mathbb{R}^{N})}^{2}=c_{0}$.
We argue by contradiction:
we assume that
$\|Q\|_{L^{2}(\mathbb{R}^{N})}^{2}\not=c_{0}$.
From Lemma \ref{Lemma5.3},
we know that
$\|Q\|_{L^{2}(\mathbb{R}^{N})}^{2}>c_{0}$.
Then
\begin{equation*}
\begin{aligned}
\frac{\|Q\|_{L^{p}(\mathbb{R}^{N})}^{p}}
{\|Q\|_{D^{s,2}(\mathbb{R}^{N})}^{\frac{2N+4-pN}{2(1-s)}}
\|Q\|_{D^{1,2}(\mathbb{R}^{N})}^{\frac{pN-2N-4s}{2(1-s)}}
\|Q\|_{L^{2}(\mathbb{R}^{N})}^{p-2}}=C_{1,s,p}
<C_{1,s,p}',
\end{aligned}
\end{equation*}
which implies that
\begin{equation*}
\begin{aligned}
\|Q\|_{L^{p}(\mathbb{R}^{N})}^{p}
<C_{1,s,p}'\|Q\|_{D^{s,2}(\mathbb{R}^{N})}^{\frac{2N+4-pN}{2(1-s)}}
\|Q\|_{D^{1,2}(\mathbb{R}^{N})}^{\frac{pN-2N-4s}{2(1-s)}}
\|Q\|_{L^{2}(\mathbb{R}^{N})}^{p-2}.
\end{aligned}
\end{equation*}
This implies that $C_{1,s,p}$ is not the best constant of inequality \eqref{1.2}.
Thus we obtain the contradiction, and the result follows.
\end{proof}

\section{Proof of Theorem \ref{Theorem1.4}}
By using \eqref{4.1}, we get that $J$ is bounded from below on $S_{c}$.
We point out that the condition $m_{c}<0$ is necessary to ensure the relative compactness (up to translations) of the minimizing sequences.
To the best of our knowledge,
the characterization result of $m_{c}$ in terms of the Gagliardo-Nirenberg inequality is novel.
\begin{lemma}\label{Lemma6.1}
Let $u\in S_{c}$, $c>0$ and $p\in(2+\frac{4s}{N},2+\frac{4}{N})$. Then $m_{c}<0$ if and only if
\begin{equation*}
\begin{aligned}
\|u\|_{D^{1,2}(\mathbb{R}^{N})}^{\frac{N(p-2)-4s}{2-2s}}
\|u\|_{D^{s,2}(\mathbb{R}^{N})}^{\frac{4-N(p-2)}{2-2s}}
<
\left[
\frac{N(p-2)-4s}{2p(1-s)}
\right]^{\frac{N(p-2)-4s}{4-4s}}
\left[
\frac{4-N(p-2)}{2p(1-s)}
\right]^{\frac{4-N(p-2)}{4-4s}}
\int_{\mathbb{R}^{N}}
|u|^{p}
\mathrm{d}x.
\end{aligned}
\end{equation*}
And $m_{c}=0$ if and only if
\begin{equation*}
\begin{aligned}
\|u\|_{D^{1,2}(\mathbb{R}^{N})}^{\frac{N(p-2)-4s}{2-2s}}
\|u\|_{D^{s,2}(\mathbb{R}^{N})}^{\frac{4-N(p-2)}{2-2s}}
\geqslant
\left[
\frac{N(p-2)-4s}{2p(1-s)}
\right]^{\frac{N(p-2)-4s}{4-4s}}
\left[
\frac{4-N(p-2)}{2p(1-s)}
\right]^{\frac{4-N(p-2)}{4-4s}}
\int_{\mathbb{R}^{N}}
|u|^{p}
\mathrm{d}x.
\end{aligned}
\end{equation*}
\end{lemma}

\begin{proof}
For any $c>0$ and $u\in S_{c}$, set
\begin{equation*}
\begin{aligned}
u_{t}(x)=t^{\frac{N}{2}}u(tx),
\end{aligned}
\end{equation*}
and
\begin{equation*}
\begin{aligned}
J(u_{t})
=\frac{t^{2}}{2}\|u\|_{D^{1,2}(\mathbb{R}^{N})}^{2}
+\frac{t^{2s}}{2}\|u\|_{D^{s,2}(\mathbb{R}^{N})}^{2}
-\frac{t^{\frac{N(p-2)}{2}}}{p}
\int_{\mathbb{R}^{N}}
|u|^{p}
\mathrm{d}x.
\end{aligned}
\end{equation*}
Set
\begin{equation*}
\begin{aligned}
h_{u}(t):=
\frac{J(u_{t})}{t^{2s}}
=\frac{t^{2-2s}}{2}\|u\|_{D^{1,2}(\mathbb{R}^{N})}^{2}
+\frac{1}{2}\|u\|_{D^{s,2}(\mathbb{R}^{N})}^{2}
-\frac{t^{\frac{N(p-2)}{2}-2s}}{p}
\int_{\mathbb{R}^{N}}
|u|^{p}
\mathrm{d}x.
\end{aligned}
\end{equation*}
\textcolor{red}{Then,}
\begin{equation*}
\begin{aligned}
h_{u}'(t)
=(1-s)t^{1-2s}\|u\|_{D^{1,2}(\mathbb{R}^{N})}^{2}
-(\frac{N(p-2)}{2}-2s)\frac{t^{\frac{N(p-2)}{2}-2s-1}}{p}
\int_{\mathbb{R}^{N}}
|u|^{p}
\mathrm{d}x.
\end{aligned}
\end{equation*}
Thus $h'_{u}(t)=0$ if and only if
\begin{equation*}
\begin{aligned}
t_{u}
=&
\left[\frac{N(p-2)-4s}{2p(1-s)}
\frac{\int_{\mathbb{R}^{N}}
|u|^{p}
\mathrm{d}x}{\|u\|_{D^{1,2}(\mathbb{R}^{N})}^{2}}\right]^{\frac{2}{4-N(p-2)}}.
\end{aligned}
\end{equation*}
From $p<2+\frac{4}{N}$, we have
$1-2s>\frac{N(p-2)}{2}-2s-1$ and
\begin{equation*}
\begin{aligned}
\begin{cases}
h'_{u}(t)<0,&t<t_{u},\\
h'_{u}(t)=0,&t=t_{u},\\
h'_{u}(t)>0,&t>t_{u}.
\end{cases}
\end{aligned}
\end{equation*}
This shows that $t_{u}$ is the unique crtical point of $h_{u}(\cdot)$, and the minimum
\begin{equation*}
\begin{aligned}
\min\limits_{t>0}h_{u}(t)=&h_{u}(t_{u})\\
=&\frac{t_{u}^{2-2s}}{2}\|u\|_{D^{1,2}(\mathbb{R}^{N})}^{2}
+\frac{1}{2}\|u\|_{D^{s,2}(\mathbb{R}^{N})}^{2}
-\frac{t_{u}^{\frac{N(p-2)}{2}-2s}}{p}
\int_{\mathbb{R}^{N}}
|u|^{p}
\mathrm{d}x\\
=&\frac{1}{2}\left[\frac{N(p-2)-4s}{2p(1-s)}
\right]^{\frac{4-4s}{4-N(p-2)}}
\frac{N(p-2)-4}{N(p-2)-4s}
\|u\|_{D^{1,2}(\mathbb{R}^{N})}^{-\frac{2N(p-2)-8s}{4-N(p-2)}}
\left[
\int_{\mathbb{R}^{N}}
|u|^{p}
\mathrm{d}x\right]^{\frac{4-4s}{4-N(p-2)}}\\
&+\frac{1}{2}\|u\|_{D^{s,2}(\mathbb{R}^{N})}^{2}.
\end{aligned}
\end{equation*}
Thus $\min\limits_{t>0}h_{u}(t)<0$ if and only if
\begin{equation*}
\begin{aligned}
&\|u\|_{D^{1,2}(\mathbb{R}^{N})}^{\frac{N(p-2)-4s}{2-2s}}
\|u\|_{D^{s,2}(\mathbb{R}^{N})}^{\frac{4-N(p-2)}{2-2s}}\\
<&\left(\left[\frac{N(p-2)-4s}{2p(1-s)}
\right]^{\frac{4-4s}{4-N(p-2)}}
\frac{4-N(p-2)}{N(p-2)-4s}
\right)^{\frac{4-N(p-2)}{4-4s}}
\|u\|_{L^{p}(\mathbb{R}^{N})}^{p}\\
=&
\left[
\frac{N(p-2)-4s}{2p(1-s)}
\right]^{\frac{N(p-2)-4s}{4-4s}}
\left[
\frac{4-N(p-2)}{2p(1-s)}
\right]^{\frac{4-N(p-2)}{4-4s}}
\|u\|_{L^{p}(\mathbb{R}^{N})}^{p}.
\end{aligned}
\end{equation*}
And $\min\limits_{t>0}h_{u}(t)\geqslant0$ if and only if
\begin{equation*}
\begin{aligned}
&\|u\|_{D^{1,2}(\mathbb{R}^{N})}^{\frac{N(p-2)-4s}{2-2s}}
\|u\|_{D^{s,2}(\mathbb{R}^{N})}^{\frac{4-N(p-2)}{2-2s}}\\
\geqslant&
\left[
\frac{N(p-2)-4s}{2p(1-s)}
\right]^{\frac{N(p-2)-4s}{4-4s}}
\left[
\frac{4-N(p-2)}{2p(1-s)}
\right]^{\frac{4-N(p-2)}{4-4s}}
\|u\|_{L^{p}(\mathbb{R}^{N})}^{p}.
\end{aligned}
\end{equation*}
\end{proof}

\begin{lemma}\label{Lemma6.2}
Let $p\in(2+\frac{4s}{N},2+\frac{4}{N})$.
If
\begin{equation*}
\begin{aligned}
c\leqslant C_{1,s,p}^{-\frac{2}{p-2}}
\left[
\frac{2N+4-pN}{2p(1-s)}
\right]^{-\frac{2N+4-pN}{2(1-s)(p-2)}}
\left[
\frac{pN-2N-4s}{2p(1-s)}
\right]^{-\frac{pN-2N-4s}{2(1-s)(p-2)}},
\end{aligned}
\end{equation*}
then $m_{c}=0$.
\end{lemma}

\begin{proof}
Using Lemma \ref{Lemma4.1}, for all $u\in S_{c}$, we know that
\begin{equation}\label{6.1}
\begin{aligned}
\|u\|_{L^{p}(\mathbb{R}^{N})}^{p}
\leqslant& C_{1,s,p}
\|u\|_{D^{s,2}(\mathbb{R}^{N})}^{\frac{2N+4-pN}{2(1-s)}}
\|u\|_{D^{1,2}(\mathbb{R}^{N})}^{\frac{pN-2N-4s}{2(1-s)}}
\|u\|_{L^{2}(\mathbb{R}^{N})}^{p-2}\\
=& C_{1,s,p}c^{\frac{p-2}{2}}
\|u\|_{D^{s,2}(\mathbb{R}^{N})}^{\frac{2N+4-pN}{2(1-s)}}
\|u\|_{D^{1,2}(\mathbb{R}^{N})}^{\frac{pN-2N-4s}{2(1-s)}}.
\end{aligned}
\end{equation}
By using \eqref{6.1} and
\begin{equation*}
\begin{aligned}
C_{1,s,p}c^{\frac{p-2}{2}}
\leqslant&
\left[
\frac{N(p-2)-4s}{2p(1-s)}
\right]^{-\frac{N(p-2)-4s}{4-4s}}
\left[
\frac{4-N(p-2)}{2p(1-s)}
\right]^{-\frac{4-N(p-2)}{4-4s}},
\end{aligned}
\end{equation*}
one deduces
\begin{equation*}
\begin{aligned}
&\|u\|_{L^{p}(\mathbb{R}^{N})}^{p}\\
\leqslant&
\left[
\frac{N(p-2)-4s}{2p(1-s)}
\right]^{-\frac{N(p-2)-4s}{4-4s}}
\left[
\frac{4-N(p-2)}{2p(1-s)}
\right]^{-\frac{4-N(p-2)}{4-4s}}
\|u\|_{D^{1,2}(\mathbb{R}^{N})}^{\frac{N(p-2)-4s}{2-2s}}
\|u\|_{D^{s,2}(\mathbb{R}^{N})}^{\frac{4-N(p-2)}{2-2s}}.
\end{aligned}
\end{equation*}
According to Lemma \ref{Lemma6.1}, we infer that $m_{c}=0$.
\end{proof}

\begin{lemma}\label{Lemma6.3}
Let $p\in(2+\frac{4s}{N},2+\frac{4}{N})$.
If $m_{c}<0$, then
\begin{equation*}
\begin{aligned}
c>C_{1,s,p}^{-\frac{2}{p-2}}
\left[
\frac{2N+4-pN}{2p(1-s)}
\right]^{-\frac{2N+4-pN}{2(1-s)(p-2)}}
\left[
\frac{pN-2N-4s}{2p(1-s)}
\right]^{-\frac{pN-2N-4s}{2(1-s)(p-2)}}.
\end{aligned}
\end{equation*}
\end{lemma}

\begin{proof}
According to Lemma \ref{Lemma6.1}, for $u\in S_{c}$, we conclude that $m_{c}<0$ if and only if
\begin{equation*}
\begin{aligned}
&
\left[
\frac{N(p-2)-4s}{2p(1-s)}
\right]^{-\frac{N(p-2)-4s}{4-4s}}
\left[
\frac{4-N(p-2)}{2p(1-s)}
\right]^{-\frac{4-N(p-2)}{4-4s}}
\|u\|_{D^{1,2}(\mathbb{R}^{N})}^{\frac{N(p-2)-4s}{2-2s}}
\|u\|_{D^{s,2}(\mathbb{R}^{N})}^{\frac{4-N(p-2)}{2-2s}}\\
<&\|u\|_{L^{p}(\mathbb{R}^{N})}^{p}.
\end{aligned}
\end{equation*}
From \eqref{6.1}, one has
\begin{equation*}
\begin{aligned}
&
\left[
\frac{N(p-2)-4s}{2p(1-s)}
\right]^{-\frac{N(p-2)-4s}{4-4s}}
\left[
\frac{4-N(p-2)}{2p(1-s)}
\right]^{-\frac{4-N(p-2)}{4-4s}}
\|u\|_{D^{1,2}(\mathbb{R}^{N})}^{\frac{N(p-2)-4s}{2-2s}}
\|u\|_{D^{s,2}(\mathbb{R}^{N})}^{\frac{4-N(p-2)}{2-2s}}\\
<&\|u\|_{L^{p}(\mathbb{R}^{N})}^{p}\\
\leqslant& C_{1,s,p}
\|u\|_{D^{s,2}(\mathbb{R}^{N})}^{\frac{2N+4-pN}{2(1-s)}}
\|u\|_{D^{1,2}(\mathbb{R}^{N})}^{\frac{pN-2N-4s}{2(1-s)}}
c^{\frac{p-2}{2}},
\end{aligned}
\end{equation*}
which implies
\begin{equation*}
\begin{aligned}
\left[
\frac{N(p-2)-4s}{2p(1-s)}
\right]^{-\frac{N(p-2)-4s}{4-4s}}
\left[
\frac{4-N(p-2)}{2p(1-s)}
\right]^{-\frac{4-N(p-2)}{4-4s}}
<
C_{1,s,p}c^{\frac{p-2}{2}}.
\end{aligned}
\end{equation*}
\end{proof}

\begin{lemma}\label{Lemma6.4}
Let $p\in(2+\frac{4s}{N},2+\frac{4}{N})$.
If
\begin{equation*}
\begin{aligned}
c> C_{1,s,p}^{-\frac{2}{p-2}}
\left[
\frac{2N+4-pN}{2p(1-s)}
\right]^{-\frac{2N+4-pN}{2(1-s)(p-2)}}
\left[
\frac{pN-2N-4s}{2p(1-s)}
\right]^{-\frac{pN-2N-4s}{2(1-s)(p-2)}},
\end{aligned}
\end{equation*}
then $m_{c}<0$.
\end{lemma}

\begin{proof}
Let $Q$ be an optimizer of the Gagliardo-Nirenberg inequality \eqref{4.1}.
Set $Q_{c}:=\frac{\sqrt{c}}{\|Q\|_{L^{2}(\mathbb{R}^{N})}}Q$.
Then
\begin{equation*}
\begin{aligned}
&\frac{
\|Q_{c}\|_{D^{s,2}(\mathbb{R}^{N})}^{\frac{2N+4-pN}{2(1-s)}}
\|Q_{c}\|_{D^{1,2}(\mathbb{R}^{N})}^{\frac{pN-2N-4s}{2(1-s)}}
\|Q_{c}\|_{L^{2}(\mathbb{R}^{N})}^{p-2}
}
{\|Q_{c}\|_{L^{p}(\mathbb{R}^{N})}^{p}}\\
=&
\frac{
\left[\frac{\sqrt{c}}{\|Q\|_{L^{2}(\mathbb{R}^{N})}}\right]^{\frac{2N+4-pN}{2(1-s)}}
\left[\frac{\sqrt{c}}{\|Q\|_{L^{2}(\mathbb{R}^{N})}}\right]^{\frac{pN-2N-4s}{2(1-s)}}
\left[\frac{\sqrt{c}}{\|Q\|_{L^{2}(\mathbb{R}^{N})}}\right]^{p-2}
}
{\left[\frac{\sqrt{c}}{\|Q\|_{L^{2}(\mathbb{R}^{N})}}\right]^{p}}\\
&\times
\frac{
\|Q\|_{D^{s,2}(\mathbb{R}^{N})}^{\frac{2N+4-pN}{2(1-s)}}
\|Q\|_{D^{1,2}(\mathbb{R}^{N})}^{\frac{pN-2N-4s}{2(1-s)}}
\|Q\|_{L^{2}(\mathbb{R}^{N})}^{p-2}
}
{\|Q\|_{L^{p}(\mathbb{R}^{N})}^{p}}\\
=&
C_{1,s,p}.
\end{aligned}
\end{equation*}
This shows that $Q_{c}$ is also an optimizer of the Gagliardo-Nirenberg inequality \eqref{4.1}.
Since $Q_{c}\in S_{c}$,
we get that
\begin{equation*}
\begin{aligned}
&\|Q_{c}\|_{L^{p}(\mathbb{R}^{N})}^{p}\\
=&C_{1,s,p}
c^{\frac{p-2}{2}}
\|Q_{c}\|_{D^{s,2}(\mathbb{R}^{N})}^{\frac{2N+4-pN}{2(1-s)}}
\|Q_{c}\|_{D^{1,2}(\mathbb{R}^{N})}^{\frac{pN-2N-4s}{2(1-s)}}\\
>&
\left[
\frac{N(p-2)-4s}{2p(1-s)}
\right]^{-\frac{N(p-2)-4s}{4-4s}}
\left[
\frac{4-N(p-2)}{2p(1-s)}
\right]^{-\frac{4-N(p-2)}{4-4s}}
\|Q_{c}\|_{D^{1,2}(\mathbb{R}^{N})}^{\frac{N(p-2)-4s}{2-2s}}
\|Q_{c}\|_{D^{s,2}(\mathbb{R}^{N})}^{\frac{4-N(p-2)}{2-2s}}.
\end{aligned}
\end{equation*}
From $Q_{c}\in S_{c}$ and Lemma \ref{Lemma6.1}, we know $m_{c}<0$.
\end{proof}

\begin{lemma}\label{Lemma6.5}
Let $p\in(2+\frac{4s}{N},2+\frac{4}{N})$.
If $m_{c}=0$,
then
\begin{equation*}
\begin{aligned}
c\leqslant C_{1,s,p}^{-\frac{2}{p-2}}
\left[
\frac{2N+4-pN}{2p(1-s)}
\right]^{-\frac{2N+4-pN}{2(1-s)(p-2)}}
\left[
\frac{pN-2N-4s}{2p(1-s)}
\right]^{-\frac{pN-2N-4s}{2(1-s)(p-2)}}.
\end{aligned}
\end{equation*}
\end{lemma}
\begin{proof}
From \cite[Theorem 1.1 (4)]{LuoTJ-Hajaiej2022ANS}, we know that $m_{c}\leqslant0$ for all $c>0$. Then we get the inverse negative proposition of Lemma \ref{Lemma6.4} without proof.
\end{proof}

\begin{proof}[Proof of Theorem \ref{Theorem1.4}]
Combining Lemmas \ref{Lemma6.1}-\ref{Lemma6.5}, we get the desired result.
\end{proof}

\section{Proof of Theorem \ref{Theorem1.5}}
In this section, we prove Theorem \ref{Theorem1.5}.
\begin{proof}[Proof of Theorem \ref{Theorem1.5}]
Let $u_{c_{0}}\in H^{1}(\mathbb{R}^{N})$ be an energy ground state solution of equation \eqref{1.1} with a critical mass $c_{0}$.
From Thoerem \ref{Theorem1.1}, we know that
\begin{equation*}
\begin{aligned}
0=J(u_{c_{0}})=m_{c_{0}}=\inf_{u\in S_{c_{0}}} J(u).
\end{aligned}
\end{equation*}
Set $u_{c_{0},t}(x)=t^{\frac{N}{2}}u_{c_{0}}(tx)$.
It follows from Lemma \ref{Lemma6.1} that
\begin{equation*}
\begin{aligned}
\min\limits_{t>0}
J(u_{c_{0},t})
=\min\limits_{t>0}
h_{u_{c_{0}}}(t)=0
=h_{u_{c_{0}}}(t_{u_{c_{0}}}),
\end{aligned}
\end{equation*}
where $t_{u_{c_{0}}}$ is the unique crtical point of $h_{u_{c_{0}}}(\cdot)$. From $J(u_{c_{0}})=0$, we further know that
\begin{equation*}
\begin{aligned}
1=t_{u_{c_{0}}}=\left[\frac{N(p-2)-4s}{2p(1-s)}
\frac{\int_{\mathbb{R}^{N}}
|u_{c_{0}}|^{p}
\mathrm{d}x}{\|u_{c_{0}}\|_{D^{1,2}(\mathbb{R}^{N})}^{2}}\right]^{\frac{2}{4-N(p-2)}},
\end{aligned}
\end{equation*}
which implies that
\begin{equation}\label{7.1}
\begin{aligned}
\|u_{c_{0}}\|_{D^{1,2}(\mathbb{R}^{N})}^{2}
=\frac{N(p-2)-4s}{2p(1-s)}
\int_{\mathbb{R}^{N}}
|u_{c_{0}}|^{p}
\mathrm{d}x.
\end{aligned}
\end{equation}
From $J(u_{c_{0}})=0$ and \eqref{7.1},
one deduces that
\begin{equation}\label{7.2}
\begin{aligned}
\|u_{c_{0}}\|_{D^{s,2}(\mathbb{R}^{N})}^{2}
=&
\frac{2}{p}
\int_{\mathbb{R}^{N}}
|u_{c_{0}}|^{p}
\mathrm{d}x
-\|u_{c_{0}}\|_{D^{1,2}(\mathbb{R}^{N})}^{2}\\
=&
\left(
\frac{4(1-s)}{2p(1-s)}
-\frac{N(p-2)-4s}{2p(1-s)}
\right)
\int_{\mathbb{R}^{N}}
|u_{c_{0}}|^{p}
\mathrm{d}x\\
=&
\frac{4-N(p-2)}{2p(1-s)}
\int_{\mathbb{R}^{N}}
|u_{c_{0}}|^{p}
\mathrm{d}x.
\end{aligned}
\end{equation}
Combining \eqref{7.1} and \eqref{7.2}, one has that
\begin{equation*}
\begin{aligned}
&\|u_{c_{0}}\|_{D^{1,2}(\mathbb{R}^{N})}^{\frac{N(p-2)-4s}{2-2s}}
\|u_{c_{0}}\|_{D^{s,2}(\mathbb{R}^{N})}^{\frac{4-N(p-2)}{2-2s}}\\
=&
\left[
\frac{N(p-2)-4s}{2p(1-s)}
\int_{\mathbb{R}^{N}}
|u_{c_{0}}|^{p}
\mathrm{d}x
\right]^{\frac{N(p-2)-4s}{4-4s}}
\left[
\frac{4-N(p-2)}{2p(1-s)}
\int_{\mathbb{R}^{N}}
|u_{c_{0}}|^{p}
\mathrm{d}x
\right]^{\frac{4-N(p-2)}{4-4s}}
\\
=&
\left[
\frac{N(p-2)-4s}{2p(1-s)}
\right]^{\frac{N(p-2)-4s}{4-4s}}
\left[
\frac{4-N(p-2)}{2p(1-s)}
\right]^{\frac{4-N(p-2)}{4-4s}}
\int_{\mathbb{R}^{N}}
|u_{c_{0}}|^{p}
\mathrm{d}x\\
=&
C_{1,s,p}^{-1}
c_{0}^{-\frac{p-2}{2}}
\int_{\mathbb{R}^{N}}
|u_{c_{0}}|^{p}
\mathrm{d}x\\
=&
C_{1,s,p}^{-1}\|u_{c_{0}}\|_{L^{2}(\mathbb{R}^{N})}^{-(p-2)}
\int_{\mathbb{R}^{N}}
|u_{c_{0}}|^{p}
\mathrm{d}x,
\end{aligned}
\end{equation*}
since $c_{0}=C_{1,s,p}^{-\frac{2}{p-2}}
\left[
\frac{2N+4-pN}{2p(1-s)}
\right]^{-\frac{2N+4-pN}{2(1-s)(p-2)}}
\left[
\frac{pN-2N-4s}{2p(1-s)}
\right]^{-\frac{pN-2N-4s}{2(1-s)(p-2)}}$ and $\|u_{c_{0}}\|_{L^{2}(\mathbb{R}^{N})}^{2}=c_{0}$.
\end{proof}

\section{Proof of Theorem \ref{Theorem1.6}}
In this section,
we prove Theorem \ref{Theorem1.6}.
Since the proof is similar to Theorems \ref{Theorem1.2}-\ref{Theorem1.5},
we just present a skeleton.
\begin{lemma}\label{Lemma8.1}
For $p\in(2+\frac{4s_{1}}{N},2+\frac{4s_{2}}{N})$ and $u\in H^{s_{2}}(\mathbb{R}^{N})$,
there exists a constant $C_{s_{2},s_{1},p}>0$ such that
\begin{equation}\label{8.1}
\begin{aligned}
\|u\|_{L^{p}(\mathbb{R}^{N})}^{p}
\leqslant C_{s_{2},s_{1},p}
\|u\|_{D^{s_{1},2}(\mathbb{R}^{N})}^{\frac{2N+4s_{2}-pN}{2(s_{2}-s_{1})}}
\|u\|_{D^{s_{2},2}(\mathbb{R}^{N})}^{\frac{pN-2N+4s_{1}}{2(s_{2}-s_{1})}}
\|u\|_{L^{2}(\mathbb{R}^{N})}^{p-2}.
\end{aligned}
\end{equation}
\end{lemma}
\begin{proof}
Taking $q=2+\frac{4s_{1}}{N}$ and  $q=2+\frac{4s_{2}}{N}$ in \eqref{2.2}, we have that
\begin{equation}\label{8.2}
\begin{aligned}
\|u\|_{L^{2+\frac{4s_{1}}{N}}(\mathbb{R}^{N})}^{2+\frac{4s_{1}}{N}}
\leqslant&
C
\|u\|_{D^{s_{1},2}(\mathbb{R}^{N})}^{2}
\|u\|_{L^{2}(\mathbb{R}^{N})}^{\frac{4s_{1}}{N}}.
\end{aligned}
\end{equation}
and
\begin{equation}\label{8.3}
\begin{aligned}
\|u\|_{L^{2+\frac{4s_{2}}{N}}(\mathbb{R}^{N})}^{2+\frac{4s_{2}}{N}}
\leqslant&
C
\|u\|_{D^{s_{2},2}(\mathbb{R}^{N})}^{2}
\|u\|_{L^{2}(\mathbb{R}^{N})}^{\frac{4s_{2}}{N}}.
\end{aligned}
\end{equation}
Applying the H\"{o}lder inequality for $p\in(2+\frac{4s_{1}}{N},2+\frac{4s_{2}}{N})$, we get that
\begin{equation*}
\begin{aligned}
\int_{\mathbb{R}^{N}}|u|^{p}\mathrm{d}x
\leqslant
\left(
\int_{\mathbb{R}^{N}}
|u|^{2+\frac{4s_{1}}{N}}
\mathrm{d}x
\right)^{\frac{(2+\frac{4s_{2}}{N})-p}{\frac{4}{N}(s_{2}-s_{1})}}
\left(
\int_{\mathbb{R}^{N}}
|u|^{2+\frac{4s_{2}}{N}}
\mathrm{d}x
\right)^{\frac{p-(2+\frac{4s_{1}}{N})}{\frac{4}{N}(s_{2}-s_{1})}},
\end{aligned}
\end{equation*}
and then using \eqref{8.2} and \eqref{8.3},
\begin{equation*}
\begin{aligned}
\int_{\mathbb{R}^{N}}|u|^{p}\mathrm{d}x
\leqslant C
\|u\|_{D^{s_{1},2}(\mathbb{R}^{N})}^{\frac{2N+4s_{2}-pN}{2(s_{2}-s_{1})}}
\|u\|_{D^{s_{2},2}(\mathbb{R}^{N})}^{\frac{pN-2N+4s_{1}}{2(s_{2}-s_{1})}}
\|u\|_{L^{2}(\mathbb{R}^{N})}^{p-2}.
\end{aligned}
\end{equation*}
\end{proof}

We define the Weinstein functional
\begin{equation*}
\begin{aligned}
\bar{W}_{p}(u)=
\frac{
\|u\|_{D^{s_{1},2}(\mathbb{R}^{N})}^{\frac{2N+4s_{2}-pN}{2(s_{2}-s_{1})}}
\|u\|_{D^{s_{2},2}(\mathbb{R}^{N})}^{\frac{pN-2N+4s_{1}}{2(s_{2}-s_{1})}}
\|u\|_{L^{2}(\mathbb{R}^{N})}^{p-2}
}{\|u\|_{L^{p}(\mathbb{R}^{N})}^{p}}
\end{aligned}
\end{equation*}
Similar to Theorem \ref{Theorem1.2},
we know that there exists $\bar{w}\in H^{s_{2}}(\mathbb{R}^{N})\backslash\{0\}$ such that
\begin{equation*}
\begin{aligned}
\|\bar{w}\|_{D^{s_{2},2}(\mathbb{R}^{N})}^{2}= \int_{\mathbb{R}^{2}}|\bar{w}|^{2}\mathrm{d}x=1,~
\bar{W}_{p}(\bar{w})=C_{s_{2},s_{1},p}^{-1},
\end{aligned}
\end{equation*}
where $\bar{w}$ is a radial non-negative minimizer for the Weinstein functional $\bar{W}_{p}$,
and $\bar{w}$ is a weak solution to the following equation
\begin{equation*}
\begin{aligned}
&(-\Delta)^{s_{2}}\bar{w}
+\frac{2N+4s_{2}-pN}{pN-2N-4s_{1}}
(-\Delta)^{s_{1}}\bar{w}
\frac{1}{\|\bar{w}\|_{D^{s_{1},2}(\mathbb{R}^{N})}^{2}}\\
&+\frac{2(p-2)(s_{2}-s_{1})}{pN-2N-4s_{1}}\bar{w}
=
\frac{2p(s_{2}-s_{1})}{pN-2N-4s_{1}}
|w|^{p-2}w
\frac{1}
{\|\bar{w}\|_{L^{p}(\mathbb{R}^{N})}^{p}}.
\end{aligned}
\end{equation*}
We define
\begin{equation}\label{8.4}
\begin{aligned}
\bar{Q}(x)=\lambda \bar{w}(\mu x),
\end{aligned}
\end{equation}
where
\begin{equation*}
\begin{aligned}
\begin{cases}
\mu=
\left[
\frac{2N+4s_{2}-pN}{pN-2N-4s_{1}}
\frac{1}{\|\bar{w}\|_{D^{s_{1},2}(\mathbb{R}^{N})}^{2}}
\right]^{\frac{1}{2(s_{1}-s_{2})}},\\
\lambda
=
\left[
\frac{2p(s_{2}-s_{1})}{pN-2N-4s_{1}}
\frac{1}{\|\bar{w}\|_{L^{p}(\mathbb{R}^{N})}^{p}}
\right]^{\frac{1}{p-2}}
\left[
\frac{2N+4s_{2}-pN}{pN-2N-4s_{1}}
\frac{1}{\|\bar{w}\|_{D^{s_{1},2}(\mathbb{R}^{N})}^{2}}
\right]^{\frac{2s_{2}}{2(s_{1}-s_{2})(p-2)}},
\end{cases}
\end{aligned}
\end{equation*}
then $\bar{Q}$ satisfies
\begin{equation*}
\begin{aligned}
(-\Delta)^{s_{2}}\bar{Q}
+
(-\Delta)^{s_{1}}\bar{Q}
+\omega \bar{Q}
=|\bar{Q}|^{p-2}\bar{Q},
\end{aligned}
\end{equation*}
with
\begin{equation*}
\begin{aligned}
\omega
=
\frac{2(p-2)(s_{2}-s_{1})}{pN-2N-4s_{1}}
\left[
\frac{2N+4s_{2}-pN}{pN-2N-4s_{1}}
\frac{1}{\|\bar{w}\|_{D^{s_{1},2}(\mathbb{R}^{N})}^{2}}
\right]^{\frac{2s_{2}}{2(s_{1}-s_{2})(p-2)}}.
\end{aligned}
\end{equation*}
Particularly, we know that $\bar{Q}$ is non-negative and radially symmetric.

Using \eqref{8.4} and the arguments in Lemmas \ref{Lemma5.1}-\ref{Lemma5.3}, we obtain that
\begin{equation*}
\begin{aligned}
\|\bar{Q}\|^{2}_{D^{s_{1},2}(\mathbb{R}^{N})}
=\frac{2N+4s_{2}-pN}{2p(s_{2}-s_{1})}
\int_{\mathbb{R}^{N}}|\bar{Q}|^{p}\mathrm{d}x,
\end{aligned}
\end{equation*}
and
\begin{equation*}
\begin{aligned}
\|\bar{Q}\|^{2}_{D^{s_{2},2}(\mathbb{R}^{N})}
=
\frac{pN-2N-4s_{1}}{2p(s_{2}-s_{1})}
\int_{\mathbb{R}^{N}}|\bar{Q}|^{p}\mathrm{d}x,
\end{aligned}
\end{equation*}
and
\begin{equation*}
\begin{aligned}
&
\frac{1}{2}
\int_{\mathbb{R}^{N}}
\int_{\mathbb{R}^{N}}
\frac{|\bar{Q}(x)-\bar{Q}(y)|^{2}}{|x-y|^{N+2s_{2}}}
\mathrm{d}x
\mathrm{d}y
+\frac{1}{2}
\int_{\mathbb{R}^{N}}
\int_{\mathbb{R}^{N}}
\frac{|\bar{Q}(x)-\bar{Q}(y)|^{2}}{|x-y|^{N+2s_{1}}}
\mathrm{d}x
\mathrm{d}y
-\frac{1}{p}
\int_{\mathbb{R}^{N}}
|\bar{Q}|^{p}
\mathrm{d}x=0.
\end{aligned}
\end{equation*}
Then repeating the steps of the proof of Theorem \ref{Theorem1.3},
one has that $\bar{Q}$ is an energy ground state solution with critical mass $\bar{c}_{0}$,
\begin{equation*}
\begin{aligned}
\int_{\mathbb{R}^{N}}
|\bar{Q}|^{p}
\mathrm{d}x
=\bar{c}_{0},
\end{aligned}
\end{equation*}
and
\begin{equation*}
\begin{aligned}
C_{s_{2},s_{1},p}
=
\left[
\frac{2N+4s_{2}-pN}{2p(s_{2}-s_{1})}
\right]^{-\frac{2N+4s_{2}-pN}{4(s_{2}-s_{1})}}
\left[
\frac{pN-2N-4s_{1}}{2p(s_{2}-s_{1})}
\right]^{-\frac{pN-2N+4s_{1}}{4(s_{2}-s_{1})}}
\frac{1}{\bar{c}_{0}^{\frac{p-2}{2}}}.
\end{aligned}
\end{equation*}
This implies that
\begin{equation*}
\begin{aligned}
\bar{c}_{0}
=
\left[
\frac{2N+4s_{2}-pN}{2p(s_{2}-s_{1})}
\right]^{\frac{2N+4s_{2}-pN}{2(s_{2}-s_{1})(p-2)}}
\left[
\frac{pN-2N-4s_{1}}{2p(s_{2}-s_{1})}
\right]^{\frac{pN-2N+4s_{1}}{2(s_{2}-s_{1})(p-2)}}
C_{s_{2},s_{1},p}^{-\frac{2}{p-2}}.
\end{aligned}
\end{equation*}

Let $u_{\bar{c}_{0}}\in H^{s_{2}}(\mathbb{R}^{N})$ be an energy ground state solution for equation \eqref{MF} with $c=\bar{c}_{0}$.
Then repeating the argument of Theorem \ref{Theorem1.5},
we have that
\begin{equation*}
\begin{aligned}
\|u_{\bar{c}_{0}}\|_{D^{s_{2},2}(\mathbb{R}^{N})}^{\frac{4-N(p-2)}{2(1-s)}}
\|u_{\bar{c}_{0}}\|_{D^{s_{1},2}(\mathbb{R}^{N})}^{\frac{N(p-2)-4s}{2(1-s)}}
=C_{s_{2},s_{1},p}^{-1}\|u_{\bar{c}_{0}}\|_{L^{2}(\mathbb{R}^{N})}^{-(p-2)}
\int_{\mathbb{R}^{N}}
|u_{\bar{c}_{0}}|^{p}
\mathrm{d}x.
\end{aligned}
\end{equation*}
Moreover, we know that $u_{\bar{c}_{0}}$ is the optimizer of inequality \eqref{1.3}.

\section*{Conflict of interest}
The author declares that he has no known competing financial interests or personal relationships that could have appeared to influence the work reported in
this article.

\section*{Data availability statement}
No data were used for the research described in the article.


\small
\begin{thebibliography}{10}




\bibitem{Biagi-Dipierro-Valdinoci-Vecchi2021CPDE}
S. Biagi, S. Dipierro, E. Valdinoci, E. Vecchi,
{Mixed local and nonlocal elliptic operators: regularity and maximum principles},
Communications in Partial Differential Equations,
\textbf{47} (2021), 585-629.

\bibitem{Biswas-Jakobsen-Karlsen2010SIAM-J-N-A}
I. Biswas, E. Jakobsen, K. Karlsen,
{Difference-quadrature schemes for nonlinear degenerate parabolic
integro-PDE},
SIAM Journal on Numerical Analysis,
\textbf{48} (2010), 1110-1135.




\bibitem{Brezis-Lieb1983PAMS}
H. Br\'{e}zis, E. Lieb,
{A relation between pointwise convergence of functions and convergence of functionals},
Proceedings of the American Mathematical Society,
\textbf{88} (1983), 486-490.


\bibitem{Burchard-Hajaiej2006JFA}
A. Burchard, H. Hajaiej,
Rearrangement inequalities for functionals with monotone integrands,
Journal of Functional Analysis,
233 (2006), 561-582.


\bibitem{Cabre-Dipierro-Valdinoci2022ARMA}
X. Cabr\'{e}, S. Dipierro, E. Valdinoci,
{The Bernstein technique for integro-differential equations},
Archive for Rational Mechanics and Analysis,
\textbf{243} (2022), 1597-1652.

\bibitem{Catto-Dolbeault-Sanchez-Soler2013MMMAS}
I. Catto, J. Dolbeault, O. Sanchez, J. Soler,
{Existence of steady states for the Maxwell-Schrodinger-Poisson system: exploring the applicability of the concentration-compactness principle},
Mathematical Models and Methods in Applied Science,
\textbf{23} (2013), 1915-1938.




\bibitem{Cangiotti-Caponi-Maione-Vitillaro2023Milan}
N. Cangiotti, M. Caponi, A. Maione, E. Vitillaro,
{Klein-Gordon-Maxwell equations driven by mixed local-nonlocal operators},
Milan Journal of Mathematics,
\textbf{91} (2023), 375-403.

\bibitem{Cangiotti-Caponi-Maione-Vitillaro2024FFAC}
N. Cangiotti, M. Caponi, A. Maione, E. Vitillaro,
{Schr\"{o}dinger-Maxwell equations driven by mixed local-nonlocal operators},
Fractional Calculus and Applied Analysis, (2024).



\bibitem{ChenHY-Bhakta-Hajaiej2022JDE}
H. Chen, M. Bhakta, H. Hajaiej,
{On the bounds of the sum of eigenvalues for a Dirichlet
problem involving mixed fractional Laplacians},
Journal of Differential Equations,
\textbf{317} (2022), 1-31.



\bibitem{ChenZQ-Kim-Song-Vondracek2012TAMS}
Z. Chen, P. Kim, R. Song, Z. Vondracek,
{Boundary Harnack principle for $\nabla+\nabla^{\alpha/2}$},
Transactions of the American Mathematical Society,
\textbf{364} (2012), 4169-4205.



\bibitem{Chergui-Gou-Hajaiej2023CVPDE}
L. Chergui, T. Gou, H. Hajaiej,
{Existence and dynamics of normalized solutions to nonlinear Schr\"odinger equations with mixed fractional Laplacians},
Calculus of Variations and Partial Differential Equations,
\textbf{62} (2023).



\bibitem{Llave-Valdinoci2009Poincare}
R. de la Llave, E. Valdinoci,
{A generalization of Aubry-Mather theory to partial differential equations and
pseudo-differential equations},
Annales de l'Institut Henri Poincar\'{e} C. Analyse Non Lin\'{e}aire,
\textbf{26} (2009), 1309-1344.



\bibitem{Dipierro-Valdinocci2021PA}
S. Dipierro, E. Valdinoci,
{Description of an ecological niche for a mixed local/nonlocal dispersal: an evolution equation and a new Neumann condition arising from the superposition of Brownian and L\'{e}vy processes},
Physica A: Statistical Mechanics and its Applications,
\textbf{575} (2021).

\bibitem{Dipierro-Lippi-Valdinocci2022AIHP}
S. Dipierro, E. Lippi, E. Valdinocci,
{(non) local logistic equations with Neumann conditions},
Annales de l'Institut Henri Poincar\'{e} C. Analyse Non Lin\'{e}aire,
\textbf{40} (2022), 1093-1166.


\bibitem{Dipierro-Su-Valdinoci-Zhang2025DCDS}
S. Dipierro, X. Su, E. Valdinoci, J. Zhang,
Qualitative properties of positive solutions of a mixed order nonlinear Schr\"{o}dinger equation,
Discrete and Continuous Dynamical Systems,
45 (2025), 1948-2000.

\bibitem{Giovannardi-Mugnai-Vecchi2023JMAA}
G. Giovannardi, D. Mugnai, E. Vecchi,
{An Ahmad-Lazer-Paul-type result for indefinite mixed local-nonlocal problems},
Journal of Mathematical Analysis and Applications, \textbf{527} (2023).


\bibitem{Hajaiej-Stuart2003PLMS}
H. Hajaiej,  C. Stuart,
Symmetrization inequalities for composition operators of Caratheodory type,
Proceedings of the London Mathematical Society,
87 (2003), 396-418.

\bibitem{Hajaiej2005Edinburgh}
H. Hajaiej,
Cases of equality and strict inequality in the extended Hardy-Littlewood inequalities,
Proceedings of the Royal Society of Edinburgh, 135 (2005), 643-661.

\bibitem{Hajaiej-Perera2022DIE}
H. Hajaiej, K. Perera,
{Ground state and least positive energy solutions of elliptic problems involving mixed fractional p-Laplacians},
Differential Integral Equations, \textbf{35} (2022), 173-190.






\bibitem{Lieb2001}
E. Lieb, M. Loss,
{Analysis}. Graduate Studies in Mathematics, 14. American Mathematical Society, Providence, RI, 2001. xxii+346 pp.

\bibitem{LuoTJ-Hajaiej2022ANS}
T. Luo, H. Hajaiej,
{Normalized solutions for a class of scalar field equations involving mixed fractional Laplacians},
Advanced Nonlinear Studies,
\textbf{22} (2022), 228-247.

\bibitem{Maione-Mugnai-Vecchi2023FFAC}
A. Maione, D. Mugnai, E. Vecchi, {Variational methods for nonpositive mixed local-nonlocal operators},
Fractional Calculus and Applied Analysis, \textbf{26} (2023), 943-961.

\bibitem{Mimica2016PLMS}
A. Mimica,
{Heat kernel estimates for subordinate Brownian motions},
Proceedings of the London Mathematical Society,
\textbf{113} (2016), 627-648.


\bibitem{Nezza-Palatucci-Valdinoci2012BSM}
E. Nezza, G. Palatucci, E. Valdinoci,
{Hitchhikers guide to the fractional Sobolev spaces},
Bulletin des Sciences Mathematiques,
\textbf{136} (2012), 521-573.

\bibitem{Su-Valdinoci-Wei-Zhang2022MZ}
X. Su, E. Valdinoci, Y. Wei, J. Zhang,
Regularity results for solutions of mixed local and nonlocal elliptic equations,
Mathematische Zeitschrift, 302 (2022), 1855-1878.

\bibitem{Su-Valdinoci-Wei-Zhang2025JDE}
X. Su, E. Valdinoci, Y. Wei, J. Zhang,
On some regularity properties of mixed local and nonlocal elliptic equations,
Journal of Differential Equations, 416 (2025), 576-613.

\bibitem{Weinstein1983CMP}
M. Weinstein,
{Nonlinear Schr\"{o}dinger Equations and Sharp Interpolation Estimates},
Communications in Mathematical
Physics,
\textbf{87} (1983), 567-576.

\bibitem{Willem1996Book}
M. Willem,
{Minimax Theorems},
Progress in Nonlinear Differential Equations and Their Applications,
\textbf{24} (1996).









\end{thebibliography}
\end{document}